\documentclass[a4paper]{amsart}
\usepackage[T1]{fontenc}
\usepackage{amssymb,amsthm,amsmath}
\usepackage[british]{babel}
\usepackage[all]{xy}
\usepackage{calrsfs}
\usepackage{graphicx}
\usepackage{hyperref}
\usepackage{mathbbol}
\usepackage{mathabx}

\theoremstyle{plain}
\newtheorem{theorem}[subsection]{Theorem}
\newtheorem{lemma}[subsection]{Lemma}
\newtheorem{proposition}[subsection]{Proposition}
\newtheorem{corollary}[subsection]{Corollary}

\theoremstyle{definition}
\newtheorem{definition}[subsection]{Definition}
\newtheorem{example}[subsection]{Example}

\theoremstyle{remark}
\newtheorem{remark}[subsection]{Remark}

\newenvironment{tfae}
{
\begin{enumerate}}
{\end{enumerate}}

\newcommand{\bigjoin}{\ensuremath{\bigvee}}
\newcommand{\bigmeet}{\ensuremath{\bigwedge}}
\newcommand{\comp}{\circ}
\newcommand{\DefEq}{\coloneq}
\newcommand{\defn}{\textbf}
\newcommand{\del}{\partial}
\newcommand{\from}{\colon}
\newcommand{\join}{\ensuremath{\vee}}
\newcommand{\meet}{\ensuremath{\wedge}}
\newcommand{\noproof}{\hfill \qed}
\newcommand{\normal}{\ensuremath{\lhd}}
\renewcommand{\square}{\raisebox{.4mm}{\,\ensuremath{\boxvoid}}}

\def\rtop{\text{\tiny\rotatebox[origin=c]{-45}{$\bot$}}}
\def\ltop{\text{\tiny\rotatebox[origin=c]{45}{$\bot$}}}

\DeclareMathOperator{\arr}{Arr}
\DeclareMathOperator{\cod}{Cod}
\DeclareMathOperator{\Coeq}{Coeq}
\DeclareMathOperator{\CoeqFork}{CoeqFork}
\DeclareMathOperator{\CoForget}{CoForget}
\DeclareMathOperator{\coker}{coker}
\DeclareMathOperator{\Coker}{Coker}
\DeclareMathOperator{\CokerSeq}{CokerSeq}
\DeclareMathOperator{\dom}{Dom}
\DeclareMathOperator{\epi}{Epi}
\DeclareMathOperator{\Eq}{Eq}
\DeclareMathOperator{\EqFork}{EqFork}
\DeclareMathOperator{\Hom}{Hom}
\DeclareMathOperator{\Forget}{Forget}
\DeclareMathOperator{\grph}{Grph}
\DeclareMathOperator{\kernel}{ker}
\renewcommand{\ker}{\kernel}
\DeclareMathOperator{\K}{Ker}
\DeclareMathOperator{\KerSeq}{KerSeq}
\DeclareMathOperator{\mono}{Mono}
\DeclareMathOperator{\op}{op}
\DeclareMathOperator{\Pull}{Pull}
\DeclareMathOperator{\Push}{Push}

\newcommand{\C}{\ensuremath{\mathbb{C}}}
\newcommand{\X}{\ensuremath{\mathbb{X}}}
\newcommand{\E}{\ensuremath{\mathcal{E}}}
\newcommand{\Y}{\ensuremath{\mathbb{Y}}}
\newcommand{\Z}{\ensuremath{\mathbb{Z}}}

\newcommand{\Ab}{\ensuremath{\mathsf{Ab}}}
\newcommand{\Arr}{\ensuremath{\mathsf{Arr}}}
\newcommand{\Arrn}{\ensuremath{\mathsf{Arr}^{n}}}

\newcommand{\DPERel}{\ensuremath{\mathsf{DPERel}}}
\newcommand{\ERel}{\ensuremath{\mathsf{ERel}}}
\newcommand{\EERel}{\ensuremath{\mathsf{EERel}}}
\newcommand{\EFork}{\ensuremath{\mathsf{EFork}}}

\newcommand{\Ext}{\ensuremath{\mathsf{Ext}}}
\newcommand{\Extn}{\ensuremath{\mathsf{Ext}^{n}}}
\newcommand{\Fork}{\ensuremath{\mathsf{Fork}}}
\newcommand{\Fun}{\ensuremath{\mathsf{Fun}}}

\newcommand{\NMono}{\ensuremath{\mathsf{NMono}}}
\newcommand{\PERel}{\ensuremath{\mathsf{PERel}}}
\newcommand{\Reg}{\ensuremath{\mathsf{Reg}}}
\newcommand{\RG}{\ensuremath{\mathsf{RGrph}}}
\newcommand{\RRel}{\ensuremath{\mathsf{RRel}}}
\newcommand{\Set}{\ensuremath{\mathsf{Set}}}
\newcommand{\Seq}{\ensuremath{\mathsf{Seq}}}
\newcommand{\ThreeCat}{\ensuremath{\mathsf{3}}}

\newdir{>>}{{}*!/3.5pt/:(1,-.2)@^{>}*!/3.5pt/:(1,+.2)@_{>}*!/7pt/:(1,-.2)@^{>}*!/7pt/:(1,+.2)@_{>}}
\newdir{ >>}{{}*!/8pt/@{|}*!/3.5pt/:(1,-.2)@^{>}*!/3.5pt/:(1,+.2)@_{>}}
\newdir{ |>}{{}*!/-3.5pt/@{|}*!/-8pt/:(1,-.2)@^{>}*!/-8pt/:(1,+.2)@_{>}}
\newdir{ >}{{}*!/-8pt/@{>}}
\newdir{>}{{}*:(1,-.2)@^{>}*:(1,+.2)@_{>}}
\newdir{<}{{}*:(1,+.2)@^{<}*:(1,-.2)@_{<}}

\def\pullback{
 \ar@{-}[]+R+<6pt,-1pt>;[]+RD+<6pt,-6pt>%
 \ar@{-}[]+D+<1pt,-6pt>;[]+RD+<6pt,-6pt>}

\def\dottedpullback{%
 \ar@{.}[]+R+<6pt,-1pt>;[]+RD+<6pt,-6pt>%
 \ar@{.}[]+D+<1pt,-6pt>;[]+RD+<6pt,-6pt>}

\title[Distributivity of congruences and the $3^n$-Lemma]{Higher extensions in exact Mal'tsev categories:\\
distributivity of congruences and the $3^n$-Lemma}

\author{Cyrille Sandry Simeu}
\author{Tim Van~der Linden}

\address[Cyrille Sandry Simeu, Tim Van~der Linden]{Institut de
Recherche en Math\'ematique et Physique, Universit\'e catholique
de Louvain, che\-min du cyclotron~2 bte~L7.01.02, B--1348
Louvain-la-Neuve, Belgium}

\email[Cyrille Sandry Simeu]{cyrille.simeu@uclouvain.be}
\email[Tim Van~der Linden]{tim.vanderlinden@uclouvain.be}

\thanks{The second author is a Research Associate of the Fonds de la Recherche Scientifique--FNRS}

\subjclass[2010]{18A20, 18G10, 18G15, 20J06, 08B10}
\keywords{$3\times 3$-Lemma; congruence distributivity; arithmetical ring, locally cyclic group; exact Mal'tsev, semi-abelian, arithmetical category; Yoneda extension; cohomology.}

\begin{document}

\begin{abstract}
The aim of this article is to better understand the correspondence between $n$-cubic extensions and $3^n$-diagrams, which may be seen as non-abelian Yoneda extensions, useful in (co)homology of non-abelian algebraic structures.

We study a higher-dimensional version of the coequaliser/kernel pair adjunction, which relates $n$-fold reflexive graphs with $n$-fold arrows in any exact Mal'tsev category. 

We first ask ourselves how this adjunction restricts to an equivalence of categories. This leads to the concept of an \emph{effective $n$-fold equivalence relation}, corresponding to the $n$-fold regular epimorphisms. We characterise those in terms of what (when $n=2$) Bourn calls \emph{parallelistic} $n$-fold equivalence relations. 

We then further restrict the equivalence, with the aim of characterising the $n$-cubic extensions. We find a congruence distributivity condition, resulting in a \emph{denormalised $3^n$-Lemma} valid in exact Mal'tsev categories. We deduce a $3^n$-Lemma for short exact sequences in semi-abelian categories, which involves a distributivity condition between joins and meets of normal subobjects. This turns out to be new even in the abelian case.
\end{abstract}

\maketitle

\section{Overview}
The classical \emph{$3\times 3$-Lemma} describes when a diagram with horizontal and vertical sequences of morphisms as in Figure~\ref{3x3 diag} may be viewed as a \emph{short exact sequence of short exact sequences}. 
\begin{figure}
$\vcenter{\xymatrix@!0@=3.5em{& 0 \ar@{->}[d] & 0 \ar@{->}[d] & 0 \ar@{->}[d]\\
0 \ar@{->}[r] & \cdot \ar@{}[rd]|{(**)} \ar@{{ |>}->}[r] \ar@{{ |>}->}[d] & \cdot \ar@{{ |>}->}[d] \ar@{-{ >>}}[r] & \cdot \ar@{{ |>}->}[d] \ar@{->}[r] & 0\\
0 \ar@{->}[r] & \cdot \ar@{{ |>}->}[r] \ar@{-{ >>}}[d] & \cdot \ar@{-{ >>}}[r] \ar@{-{ >>}}[d] \ar@{}[rd]|{(*)} & \cdot \ar@{-{ >>}}[d] \ar@{->}[r] & 0\\
0 \ar@{->}[r] & \cdot \ar@{{ |>}->}[r]_-k \ar[d] & \cdot \ar@{-{ >>}}[r]_-f \ar@{->}[d] & \cdot \ar@{->}[d] \ar@{->}[r] & 0\\
& 0 & 0 & 0 }}$
\caption{A $3\times 3$-diagram: all rows and columns are exact sequences.}\label{3x3 diag}
\end{figure}
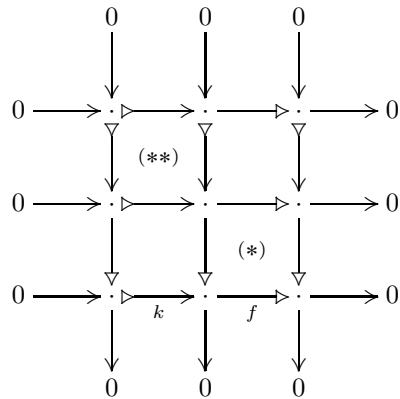
The aim of this article is to explain how to continue this process: to describe what is a $3\times 3 \times 3 $-diagram---in some sense, a short exact sequence between $3\times 3$-diagrams---and so on; see Figure~\ref{Figure 3x3x3} for a picture when~${n=3}$. This leads to a \emph{$3^n$-Lemma} for each $n\geq 2$, which extends the $3\times 3$-Lemma of~\cite{Bourn2001} to higher degrees.

We shall see that, in a semi-abelian category~\cite{Janelidze-Marki-Tholen}, the concept of a higher (cubic) extension (in the sense of~\cite{EGVdL, EGoeVdL, RVdL2, PVdL1} and the papers referred to there) is equivalent to the notion of a \emph{$3^n$-diagram} introduced here. These may be understood as a non-abelian version of the concept of a Yoneda extension~\cite{Yoneda-Exact-Sequences}, which is useful for instance when studying cohomology of non-abelian algebraic structures~\cite{RVdL2,PVdL1}.

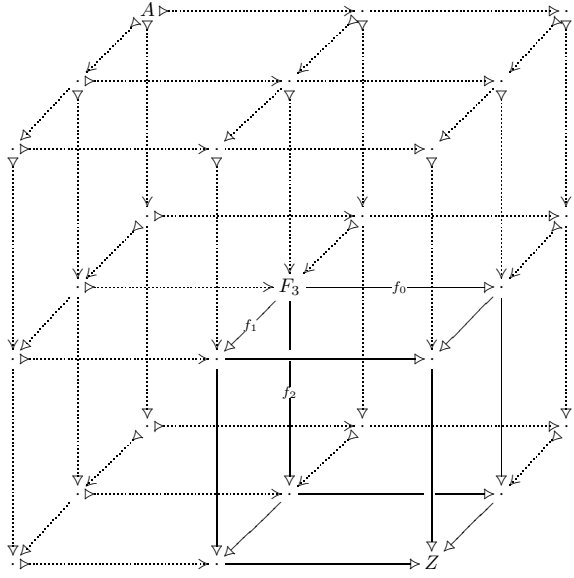
\begin{figure}
\resizebox{.6\textwidth}{!}{
$\vcenter{\xymatrix{&& A \ar@{{ |>}.>}[ld] \ar@{{ |>}.>}[rrr] \ar@{{ |>}.>}[ddd] &&& \cdot \ar@{{ |>}.>}[ld] \ar@{{ |>}.>}[ddd] \ar@{.{ >>}}[rrr] &&& \cdot \ar@{{ |>}.>}[ld] \ar@{{ |>}.>}[ddd]\\
&{\cdot} \ar@{.{ >>}}[ld] \ar@{{ |>}.>}[rrr] \ar@{{ |>}.>}[ddd] &&& \cdot \ar@{.{ >>}}[ld] \ar@{{ |>}.>}[ddd] \ar@{.{ >>}}[rrr] &&& \cdot \ar@{.{ >>}}[ld] \ar@{{ |>}.>}[ddd]\\
{\cdot} \ar@{{ |>}.>}[rrr] \ar@{{ |>}.>}[ddd] &&& \cdot \ar@{{ |>}.>}[ddd] \ar@{.{ >>}}[rrr] &&& \cdot \ar@{{ |>}.>}[ddd]\\
&&{\cdot} \ar@{{ |>}.>}[ld] \ar@{{ |>}.>}[rrr] \ar@{.{ >>}}[ddd] &&& \cdot \ar@{{ |>}.>}[ld] \ar@{.{ >>}}[ddd] \ar@{.{ >>}}[rrr] &&& \cdot \ar@{{ |>}.>}[ld] \ar@{.{ >>}}[ddd]\\
&{\cdot} \ar@{.{ >>}}[ld] \ar@{{ |>}.>}[rrr] \ar@{.{ >>}}[ddd] &&& F_{3} \ar@{-{ >>}}[ld]|-{f_{1}} \ar@{-{ >>}}[ddd]|(.33){\hole}|-{f_{2}} \ar@{-{ >>}}[rrr]|-{f_{0}} &&& \cdot \ar@{-{ >>}}[ld] \ar@{-{ >>}}[ddd]\\
{\cdot} \ar@{{ |>}.>}[rrr] \ar@{.{ >>}}[ddd] &&& \cdot \ar@{-{ >>}}[ddd] \ar@{-{ >>}}[rrr] &&& \cdot \ar@{-{ >>}}[ddd]\\
&&{\cdot} \ar@{{ |>}.>}[ld] \ar@{{ |>}.>}[rrr] &&& \cdot \ar@{{ |>}.>}[ld] \ar@{.{ >>}}[rrr] &&& \cdot \ar@{{ |>}.>}[ld] \\
&{\cdot} \ar@{.{ >>}}[ld] \ar@{{ |>}.>}[rrr] &&& \cdot \ar@{-{ >>}}[ld] \ar@{-{ >>}}[rrr]|(.66){\hole} &&& \cdot \ar@{-{ >>}}[ld]\\
{\cdot} \ar@{{ |>}.>}[rrr] &&& \cdot \ar@{-{ >>}}[rrr] &&& Z}}$}
\caption{A $3\times 3\times 3$-diagram: all pairs of arrows are short exact sequences.}\label{Figure 3x3x3}
\end{figure}

In order to obtain a \emph{pointed} version of the $3^n$-Lemma, valid in semi-abelian categories, we first need to analyse the concept of a cubic extension in an exact Mal'tsev context, where we shall prove a \emph{denormalised $3^n$-Lemma}, which deals with exact forks (certain augmented reflexive graphs) instead of short exact sequences (see Figure~\ref{denorm 3x3 diag}). This is meant to be a higher-dimensional version of the \emph{denormalised $3\times 3$-Lemma} of~\cite{Bourn2003}. In other words, our analysis depends on an investigation of cubic extensions---certain commutative cubes of arrows, see Section~\ref{Section Introduction} where the definition is recalled---via internal (higher) equivalence relations.

Let us now present a brief sketch of the structure of the paper, again leaving certain concepts (which will be introduced in Section~\ref{Section Introduction}) undefined. First we introduce the notion of an \emph{effective} $n$-fold equivalence relation (on~$n$ equivalence relations $R_i$, $0\leq i< n$, on an object $X$), which is an equivalence relation that via the coequaliser/kernel pair equivalence corresponds to an $n$-fold regular epimorphism. As it turns out, an $n$-fold equivalence relation is effective if and only if it is \emph{parallelistic}: it is isomorphic to the largest $n$-fold equivalence relation on $(R_i)_{i\in n}$, which is of the form~$\bigboxvoid_{i\in n}R_i$ described in~\cite{RVdL2}. We see that, for any $n$-tuple $(R_i)_{i\in n}$ of equivalence relations on~$X$,
\begin{enumerate}
	\item an $n$-fold effective equivalence relation over them always exists, namely the $n$-fold parallelistic equivalence relation $\bigboxvoid_{i\in n}R_i$;
	\item the $n$-fold regular epimorphism induced by taking pushouts of coequalisers need, however, not be a cubic extension in general;
	\item hence it is not always possible to construct an $n$-cubic extension by pushouts of coequalisers out of the given equivalence relations $R_i$.
\end{enumerate}
We are thus confronted with the following question:
\begin{quote}
\emph{When does a finite collection of equivalence relations $(R_i)_{i\in n}$\\ induce a cubic extension?} 	
\end{quote} 
Our main target in this paper is to explain that the answer is a distributivity condition: if $J_0$, $J_1$, \dots, $J_k\subseteq n$ with $k\geq 1$ such that
$J_i \cap J_j=\emptyset$ for all $i\neq j$, then
\[
 \big(\bigmeet_{j\in J_0}R_j\big)\meet \bigjoin_{i=1}^k\big(\bigmeet_{j\in J_i}R_j\big)=
 \bigjoin_{i=1}^k\big( \bigmeet_{j\in J_0\cup J_i}R_j\big).
\]
This condition characterises higher cubic extensions amongst higher regular epimorphisms in terms of the kernel pairs $R_i=\Eq(f_i)$ of the ``initial ribs'' $f_i$ of a higher regular epimorphism $F$. As a special case, we regain the result from~\cite{EGJVdL} that, when in a given category $\X$ all double regular epimorphisms are double (=~$2$-cubic) extensions, then $\X$ is congruence distributive. That is to say, it is \defn{arithmetical} in the sense of~\cite{Pedicchio2}.

 Reinterpreting the distributivity condition in the pointed context of a semi-abelian category, we may answer the question under which conditions a collection of $n$ normal subobjects $(K_i)_{0\leq i< n}$ on an object $X$ induces a $3^n$-diagram, and thus an $n$-cubic extension, by first taking intersections, and then cokernels of those intersections. From the above we deduce that this happens if and only if the equality
\[
 \big(\bigmeet_{j\in J_0}K_j\big)\meet \bigjoin_{i=1}^k\big(\bigmeet_{j\in J_i}K_j\big)=
 \bigjoin_{i=1}^k\big( \bigmeet_{j\in J_0\cup J_i}K_j\big)
\]
holds whenever $J_0$, $J_1$, \dots, $J_k\subseteq n$ with $k\geq 1$ such that
$J_i \cap J_j=\emptyset$ for all $i\neq j$.

Note that even in an abelian category, this condition cannot come for free. The reason is that also here, if all collections of normal monomorphisms/congruences are distributive, then the category is arithmetical---however, the only arithmetical abelian category is the trivial one~\cite{Pedicchio2}. This may explain why in the literature, as far as we know, currently $3\times 3\times 3$-diagrams were not even considered in the case of modules over a ring.

The following section explains the main ideas of the text in detail, sketching the necessary background and terminology. Section~\ref{Section Exact Mal'tsev} is where the real work is done: proving the denormalised $3^n$-Lemma in an exact Mal'tsev context. We actually prove three different versions of this result, Theorem~\ref{3^n iff extension}, Theorem~\ref{3^n iff distributive parallelistic} and Theorem~\ref{3^n overview}. In Section~\ref{Section Semiabelian} we apply this in the context of semi-abelian categories in order to obtain Theorem~\ref{Theorem 3^n}: the $3^n$-Lemma. We end the paper with some final remarks made in Section~\ref{Section Final}.

\section{Introduction}\label{Section Introduction}
In this section we elaborate on the concepts mentioned in the previous section, recalling definitions and results from the literature, introducing some new notions.

\subsection{The pointed $3\times 3$-Lemma}
In his article~\cite{Bourn2001}, Dominique Bourn proved that the classical $3\times 3$-Lemma, well known to be valid for algebraic structures such as groups and modules, may be extended to pointed regular protomodular categories (i.e., pointed and regular categories where the \emph{Split Short Five Lemma} holds). In the context of a semi-abelian category (pointed protomodular, Barr exact with binary coproducts), it amounts to the following: the diagram in Figure~\ref{3x3 diag} is a \defn{$3\times 3$-diagram} when all of its rows and columns are \emph{exact sequences}, which for the bottom horizontal sequence means that $k=\ker(f)$, $f=\coker(k)$. (In other words, the couple of arrows $(k,f)$ is a \emph{short exact sequence}.) The \defn{$3\times 3$-Lemma} gives necessary and sufficient conditions for a given diagram to be a $3\times 3$-diagram. For instance, if all rows and two out of three of the columns are exact, while the middle column is null, then the third column is exact as well~\cite[Theorem~12]{Bourn2001}. 

An important question in practice is how to construct a $3\times 3$-diagram out of a given commutative square $(*)$ or $(**)$ by taking kernels or cokernels, respectively. It is well known that in a semi-abelian category, the square $(**)$ induces a $3\times 3$-diagram by taking cokernels if and only if it is a pullback square, all of whose arrows are normal monomorphisms (=~kernels). On the other hand, $(*)$ induces a $3\times 3$-diagram by taking kernels if and only if it is a pushout square, all of whose arrows are normal epimorphisms (=~cokernels). To see this, it suffices for instance to combine Lemma~4.2.5 in~\cite{Borceux-Bourn} with Proposition~3.3 in~\cite{EGVdL} and the explanation given there.

\subsection{The denormalised $3\times 3$-Lemma}\label{Subsection Denorm}
In \cite{Bourn2003}, Dominique Bourn extended the above to a non-pointed setting. In the context of a regular category, a \defn{(denormalised) $3\times 3$-diagram} is a diagram as in Figure~\ref{denorm 3x3 diag}, where all rows and columns are \defn{exact (reflexive) forks}, which for the bottom horizontal sequence means that the reflexive graph $(c,d)$ is the kernel relation of $f$, and $f$ is the coequaliser of $d$ and $c$.
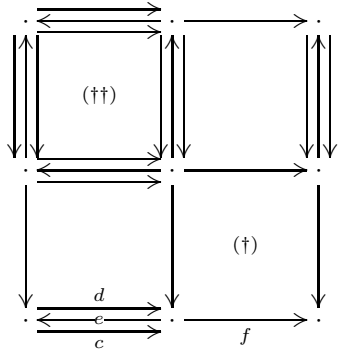
\begin{figure}
$\vcenter{\xymatrix@=2em{\cdot\ar@{}[rrdd]|-{(\dagger\dagger)} \ar@<1ex>[rr]\ar@<-1ex>[rr]\ar@<1ex>[dd]\ar@<-1ex>[dd]_{}&&\cdot \ar[ll]\ar@<1ex>[dd]^{}\ar@<-1ex>[dd]_{}\ar[rr]&&\cdot\ar@<1ex>[dd]\ar@<-1ex>[dd]\\
&&&&\\
\cdot \ar[uu]\ar@<1ex>[rr]\ar@<-1ex>[rr]\ar[dd]&&\cdot\ar@{}[rrdd]|-{(\dagger)}\ar[ll]\ar[uu]\ar[dd]\ar[rr] &&\cdot \ar[uu]\ar[dd]^{}\\
&&&&\\
\cdot\ar@<1ex>[rr]^-d\ar@<-1ex>[rr]_-{c}&&\cdot\ar[ll]|-{e}\ar[rr]_-f&&\cdot }}$
\caption{A denormalised $3\times 3$-diagram: all rows and columns are exact forks.}\label{denorm 3x3 diag}
\end{figure}
The \defn{denormalised $3\times 3$-Lemma} gives necessary and sufficient conditions for a given diagram of arrows and reflexive graphs as in Figure~\ref{denorm 3x3 diag} to be a $3\times 3$-diagram. For instance, if all rows and two out of three of the columns including the middle one are exact forks, then the third column is an exact fork as well~\cite[Theorem~3.1]{Bourn2003}. However, this only works if the context is sufficiently strong: the article~\cite{Bourn2003} treats the case of a regular Mal'tsev category, but variations on this theme have been considered in more general environments~\cite{Lack, ZJanelidze-Snake, Gran-Rodelo}; in~\cite{GJR}, the pointed and unpointed cases are even studied in a single framework.

As in the pointed case, we are interested in characterising when a square $(\dagger)$ or a double reflexive graph $(\dagger\dagger)$ as in Figure~\ref{denorm 3x3 diag} induces a $3\times 3$-diagram by taking kernel pairs or coequalisers, respectively. In the article~\cite{Bourn2003}, in the context of a Barr exact Mal'tsev category $\X$ it is proved that $(\dagger\dagger)$ induces a $3\times 3$-diagram by taking coequalisers if and only if it is a so-called \emph{parallelistic double equivalence relation}. Given two equivalence relations $R$ and $S$ on an object~$X$, a double equivalence relation over them is said to be \defn{parallelistic} when it is isomorphic to the largest double equivalence relation $D$ on $R$ and $S$, which is written $R\square S$ as in~Figure~\ref{Figure Square}.
\begin{figure}
$\vcenter{\xymatrix@=2em{R\square S \ar@<1ex>[rr]^-{\pi_2^S} \ar@<-1ex>[rr]_-{\pi_1^S}\ar@<1ex>[dd]^-{\pi_2^R}\ar@<-1ex>[dd]_-{\pi_1^S}&&S\ar[ll]\ar@<1ex>[dd]^-{s_2}\ar@<-1ex>[dd]_-{s_1}\\\\
R \ar[uu]\ar@<1ex>[rr]^-{r_1}\ar@<-1ex>[rr]_-{r_2}&& X\ar[ll]\ar[uu]}}
\qquad
\vcenter{\xymatrix@=4em{R\square
S \pullback \ar[r]^-{\langle\pi_{1}^{S},\pi_{2}^{S}\rangle} \ar[d]_-{\langle\pi_{1}^{R},\pi_{2}^{R}\rangle} & S \times S\ar[d]^-{\langle s_1,s_2\rangle\times\langle s_1,s_2\rangle}\\
R\times R\ar[r]_-{\langle r_1\times r_1
,r_2\times r_2\rangle}& X^4 }}$
\caption{The parallelistic double equivalence relation $R\square S$.}\label{Figure Square}
\end{figure}
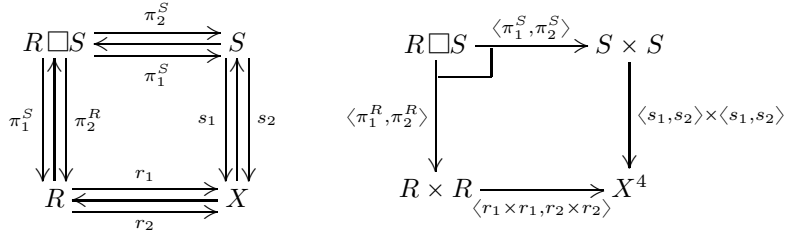
In other words, the forgetful functor $U\colon\ERel^2(\X)\to \ERel(\X)\times \ERel(\X)$ that sends a double equivalence relation~$D$ over $R$ and $S$ to the couple $(R,S)$ has a right adjoint
\[
\xymatrix{\ERel^2(\X) \ar@<1ex>[r]^-{U} \ar@{}[r]|-{\bot} & \ERel(\X)^2, \ar@<1ex>[l]^-{\square}}
\]
which takes two equivalence relations $R$ and $S$ and sends them to the parallelistic double equivalence relation~$R\square S$. It is obtained via the pullback in Figure~\ref{Figure Square}; the ``elements'' of $R\square S$ are quadruples $(x, y, t, z)\in X^4$ such that
\[
\vcenter{\xymatrix@!0@=2em{x \ar@{}[d]|{ R} \ar@{}[r]|{ S} & y \ar@{}[d]|{ R}\\
t \ar@{}[r]|{ S} & z}}
\qquad\qquad\text{$xSy$, $tSz$, $xRt$ and $yRz$.}
\]

On the other hand, a square $(\dagger)$ as in Figure~\ref{denorm 3x3 diag} induces a $3\times 3$-diagram by taking kernel pairs if and only if it is a \defn{regular pushout square} or a \defn{double} (or \defn{$2$-cubic}) \defn{extension}, which means that all of its arrows, as well as the induced comparison to the pullback, are regular epimorphisms (=~coequalisers of some pair of parallel arrows)---see Figure~\ref{Figure Regular Pushout}. 
\begin{figure}
$\vcenter{\xymatrix{A_1\ar@{-{ >>}}@/^/[rrd]^-{f_1} \ar@{-{ >>}}@/_/[rdd]_-a \ar@{-{ >>}}[rd]|-{\langle a,f_1\rangle} &&\\
&A_0\times_{B_0}B_1 \pullback\ar@{-{ >>}}[r]\ar@{-{ >>}}[d] &B_1\ar@{-{ >>}}[d]^-b\\
&A_0\ar@{-{ >>}}[r]_-{f_0} &B_0}}$
\caption{The outer square is a \emph{regular pushout} when all arrows in the induced diagram are regular epimorphisms.}\label{Figure Regular Pushout}
\end{figure}
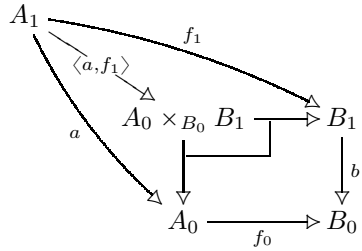
In general, pushouts and regular pushouts do not coincide; by Theorem~5.7 in~\cite{Carboni-Kelly-Pedicchio}, a regular category is an exact Mal'tsev category precisely when every pushout of two regular epimorphisms is a regular pushout. So in the context of semi-abelian categories we regain the characterisation recalled above: the normalisation of an equivalence relation is a normal monomorphism, and a double equivalence relation normalises to a pullback of normal monomorphisms.

\subsection{Higher regular epimorphisms, higher cubic extensions} 
Our aim is now to extend this to higher degrees. It is, however, crucial here to recall~\cite{EGVdL} that even in a semi-abelian category, $n$-fold regular epimorphisms do not need to satisfy the extension condition when~${n\geq 3}$.

In a regular category~$\X$, one-cubic extensions are just regular epimorphisms, which, in the varietal case, are exactly the surjective morphisms. We write $\Reg(\X)$ for the full subcategory of the 
category of arrows $\Arr(\X)$ in~$\X$ determined by the regular epimorphisms. It is well known and easily seen that this is again a regular category, whose regular epimorphisms, viewed in the base category~$\X$, are pushout squares of regular epimorphisms. Repeating this process inductively we find the full subcategory $\Reg^n(\X)=\Reg(\Reg^{n-1}(\X))$ of the category of $n$-fold 
arrows $\Arrn(\X)=\Arr(\Arr^{n-1}(\X))$ determined by the $n$-fold regular epimorphisms. Viewed in $\X$, its objects are $n$-cubes, all of whose squares are pushouts of regular epimorphisms.

We recall the inductive definition of a higher cubic extension. (Our terminology here follows~\cite{PVdL1}, where in accordance with~\cite{Yoneda-Exact-Sequences}, ``extension'' means ``short exact sequence'' rather than ``regular epimorphism''; we here call ``cubic extension'' what is called ``extension'' in~\cite{EGVdL, EGoeVdL} and elsewhere.) Denoting by $\E$ the class of cubic extensions in~$\X$, a \defn{$2$-cubic extension} (also called a \defn{double extension}) in $\X$ is a commutative square as in Figure~\ref{Figure Regular Pushout} where the morphisms $a$, $b$, $f_1$, $f_0$ and the universally induced morphism $\langle a,f_1\rangle\colon{A_1\to A_0\times_{B_0} B_1}$ to the pullback of $b$ and $f_0$ are in $\E$. We denote the class of $2$-cubic extensions thus obtained by $\E^{1}$. Of course this definition does not depend on the exact nature of one-dimensional extensions, so it can be used for any (reasonable) class of morphisms $\E$. In particular, it can be iterated to give $n$-cubic extensions for any $n\geq 2$: then all the arrows in the induced diagram are $(n-1)$-cubic extensions. We write $\Ext(\X)$ for the full subcategory of the category of arrows $\Arr(\X)$ in $\X$ determined by the extensions (=~$\Reg(\X)$), and 
similarly~$\Extn(\X)$ for the full subcategory of $\Arrn(\X)$ determined by the $n$-cubic extensions.

\subsection{Arrows versus reflexive graphs}

Let $\X$ be a category in which every arrow has a kernel pair and every reflexive graph has a coequaliser. We consider the basic coequaliser/kernel pair adjunction
\[
\xymatrix{\RG(\X) \ar@<1ex>[r]^-{\Coeq} \ar@{}[r]|-{\bot} & \Arr(\X) \ar@<1ex>[l]^-{\Eq}}
\]
between the category $\RG(\X)$ of reflexive graphs and the category $\Arr(\X)$ of arrows in $\X$. The left adjoint sends a reflexive graph $(G,d,c,e)$
\[
\xymatrix{G \ar@<1ex>[r]^-{d} \ar@<-1ex>[r]_-{c} & X \ar[l]|-{e} } \qquad\qquad d\comp e=1_X=c\comp e
\]
to the coequaliser $\Coeq(d,c)$ of $d$ and $c$, while the right adjoint sends an arrow~$f\colon{X\to Y}$ to its kernel relation $(\Eq(f), \pi_1,\pi_2,\Delta_X)$. 

We shall be concerned with restricting this adjunction to an adjoint equivalence. By definition, an arrow is a regular epimorphism if and only if it is the coequaliser of some parallel pair of maps. Equivalently, it is the coequaliser of its kernel pair. So the image of the left adjoint $\Coeq$ is the full subcategory $\Reg(\X)$ of $\Arr(\X)$ determined by the regular epimorphisms. On the other hand, a reflexive graph is in the image of the functor $\Eq$ precisely when it is an effective equivalence relation. Writing $\EERel(\X)$ for the category of effective equivalence relations in $\X$, the above adjunction restricts to an equivalence of categories
\[
\xymatrix{\EERel(\X) \ar@<1ex>[r]^-{\Coeq} \ar@{}[r]|-{\simeq} & \Reg(\X). \ar@<1ex>[l]^-{\Eq}}
\]
When $\X$ is Barr exact, we may take the category $\ERel(\X)$ of equivalence relations in $\X$ on the left; and when $\X$ is, moreover, a Mal'tsev category, we may take the category $\RRel(\X)$ of reflexive relations instead. When $\X$ is a regular category, the existence of the respective equivalence characterises when $\X$ is Barr exact (equivalence relations are effective) or Barr exact Mal'tsev (reflexive relations are effective equivalence relations). Note that in a Barr exact Mal'tsev category, any reflexive graph has a coequaliser, since this coequaliser may be computed as the coequaliser of the support of the graph, which is a reflexive relation, hence an effective equivalence relation.

\subsection{Higher arrows versus higher reflexive graphs}\label{Subsection Higher Arrows}
Our next aim is to extend this adjunction, and the induced equivalence of categories, to higher arrows and higher reflexive graphs. In the case of double arrows and double reflexive graphs we may compose adjunctions as in
\[
\xymatrix{\RG(\RG(\X)) \ar@<1ex>[r]^-{\Coeq} \ar@{}[r]|-{\bot} & \Arr(\RG(\X)) \ar@<1ex>[l]^-{\Eq} \ar@<1ex>[r]^-{\Arr(\Coeq)} \ar@{}[r]|-{\bot} & \Arr(\Arr(\X)) \ar@<1ex>[l]^-{\Arr(\Eq)}}
\]
in order to obtain an adjunction
\[
\xymatrix{\RG^2(\X) \ar@<1ex>[r]^-{\Coeq^2} \ar@{}[r]|-{\bot} & \Arr^2(\X) .\ar@<1ex>[l]^-{\Eq^2}}
\]
Here we write $\Arr\colon{\mathsf{CAT}\to \mathsf{CAT}}$ for the functor which sends a category $\X$ to the category of arrows $\Arr(\X)$, and a functor $F\colon{X\to Y}$ to the naturally induced functor $\Arr(F)\colon{\Arr(\X)\to \Arr(\Y)}$. The left adjoint $\Coeq^2$ takes a double reflexive graph and sends it to the coequaliser of its coequaliser, whereas the right adjoint $\Eq^2$ takes a double arrow and sends it to the kernel pair of its kernel pair. (See Figure~\ref{denorm 3x3 diag}, where we can make the arbitrary choice of viewing the objects of $\Arr(\X)$ as vertical arrows, and the arrows between those horizontally.) A square is in the image of $\Coeq^2$ when it is a double regular epimorphism: a pushout square whose arrows are regular epimorphisms. On the other hand, the image of $\Eq^2$ consists of what we shall call \defn{double effective equivalence relations}. Thus the adjunction restricts to an equivalence 
\[
\xymatrix{\EERel^2(\X) \ar@<1ex>[r]^-{\Coeq^2} \ar@{}[r]|-{\simeq} & \Reg^2(\X). \ar@<1ex>[l]^-{\Eq^2}}
\]
As explained above, when a double equivalence relation is effective was essentially characterised in~\cite{Bourn2003}: they are the so-called \emph{parallelistic double equivalence relations}. In a regular Mal'tsev category, any double effective equivalence relation is uniquely determined by the underlying effective equivalence relations $R$ and $S$ as the double equivalence relation $R\square S$, and any choice or $R$ and $S$ determines one such (see Figure~\ref{Figure Square}). This may be explained by the fact that the left adjoint $\Coeq^2\colon {\ERel^2(\X)\to \Reg^2(\X)}$ can be written as a composite of left adjoints like in Figure~\ref{Figure Adjunction}.
\begin{figure}
\[
\xymatrix@!0@=8em{\ERel^2(\X) \ar@<1ex>[r]^-{U} \ar@{}[r]|-{\bot} & \ERel(\X)^2 \ar@<1ex>[l]^-{\square}
\ar@<1ex>[r]^-{(\Coeq)^2} \ar@{}[r]|-{\bot} & \Reg(\X)^2 \ar@<1ex>[l]^-{(\Eq)^2} \ar@<1ex>[r]^-{\Push} \ar@{}[r]|-{\simeq} & \ar@<1ex>[l]^-{\Forget} \Reg^2(\X)}
\]
\[
\vcenter{\xymatrix@!0@=4em{\cdot \ar@<1ex>[r] \ar@<-1ex>[r]\ar@<1ex>[d]\ar@<-1ex>[d] & S\ar[l]\ar@<1ex>[d]\ar@<-1ex>[d]
& & S\ar@<1ex>[d]\ar@<-1ex>[d]
& X \ar@{-{ >>}}[r] \ar@{-{ >>}}[d] & X/_R 
& X \ar@{-{ >>}}[r] \ar@{-{ >>}}[d] & X/_R \ar@{-{ >>}}[d]\\
R \ar[u]\ar@<1ex>[r]\ar@<-1ex>[r]& X\ar[l]\ar[u]
& R \ar@<1ex>[r]\ar@<-1ex>[r]& X\ar[l]\ar[u]
& X/_S & 
& X/_S \ar@{-{ >>}}[r] & \cdot}}
\]
\caption{For double equivalence relations, parallelistic = effective.}\label{Figure Adjunction}
\end{figure}
Since the two left adjoints coincide by construction, the two right adjoints are naturally isomorphic. In particular, they have the same replete image.

As already mentioned above, it follows from~\cite[Theorem~5.7]{Carboni-Kelly-Pedicchio} that $\X$ is Barr exact Mal'tsev if and only if every double regular epimorphism is a double extension. So in that context, through the above equivalence, double extensions admit a characterisation in terms of parallelistic double equivalence relations. Outside the exact Mal'tsev context, however, the category $\Ext^2(\X)$ is strictly smaller than $\Reg^2(\X)$, so if we want the equivalence above to describe double extensions in terms of double equivalence relations, then it must be restricted to a smaller category on the left. As explained by Proposition~5.4 in~\cite{Carboni-Kelly-Pedicchio}, we need to add the condition that the join $R\join S=RS=SR$ is an effective equivalence relation. We do, however, not wish to take into account right now these kinds of difficulties having to do with working in a non-exact context, so from now on we shall restrict our analysis to Barr exact Mal'tsev categories.

The construction above may be repeated to yield an adjunction 
\[
\xymatrix{\RG^n(\X) \ar@<1ex>[r]^-{\Coeq^n} \ar@{}[r]|-{\bot} & \Arr^n(\X)\ar@<1ex>[l]^-{\Eq^n}}
\]
for any $n\geq1$: it suffices to put
\[
\Coeq^n=\Arr(\Coeq^{n-1}) \comp \Coeq=\Arr^{n-1}(\Coeq)\comp \cdots \comp \Arr(\Coeq) \comp \Coeq
\]
and
\[
\Eq^n=\Eq\comp \Arr(\Eq^{n-1})=\Eq\comp \Arr(\Eq)\comp\cdots\comp\Arr^{n-1}(\Eq),
\]
which are adjoint by composition of adjoints. As above, the image of $\Coeq^n$ is the category of $n$-fold regular epimorphism, whereas the image of $\Eq^n$ consists of what we shall call \defn{$n$-fold effective equivalence relations}. We thus obtain the equivalence 
\[
\xymatrix{\EERel^n(\X) \ar@<1ex>[r]^-{\Coeq^n} \ar@{}[r]|-{\simeq} & \Reg^n(\X). \ar@<1ex>[l]^-{\Eq^n}}
\]
In particular, an $n$-fold reflexive graph $G$ is an effective $n$-fold equivalence relation if and only if $G=\Eq^n(\Coeq^n(G))$; an $n$-fold arrow $F$ is an $n$-fold regular epimorphism if and only if $F=\Coeq^n(\Eq^n(F))$.

One interesting question which arises naturally in this context is to characterise the $n$-fold effective equivalence relations in an exact Mal'tsev category. Here we may extend the case $n=2$ by introducing the concept of an \emph{$n$-fold parallelistic equivalence relation} as in~\cite{RVdL2}. We follow a quick pragmatic course: we have the forgetful functor $U$ which takes an $n$-fold equivalence relation on $R_0$, \dots, $R_{n-1}$ and sends it to $(R_0,\dots, R_{n-1})$; its left adjoint
\[
\xymatrix{\ERel^n(\X) \ar@<1ex>[r]^-{U} \ar@{}[r]|-{\bot} & \ERel(\X)^n \ar@<1ex>[l]^-{\square}}
\]
takes an $n$-tuple of equivalence relations $(R_i)_{i\in n}$ and sends it to the $n$-fold equivalence relation $\bigboxvoid_{i\in n}R_i$. If we call an $n$-fold equivalence relation \defn{parallelistic} when it is in the replete image of this left adjoint, then the adjunction restricts to an equivalence between the category $\ERel(\X)^n$ of $n$-tuples of equivalence relations and the category $\PERel^n(\X)$ of parallelistic $n$-fold equivalence relations in $\X$. We may now view the functor $\Coeq^n\colon{\ERel^n(\X)\to \Reg^n(\X)}$ as the composite of left adjoints
\[
\xymatrix{\ERel^n(\X) \ar[r]^-U & \ERel(\X)^n \ar[r]^-{(\Coeq)^n} & \Reg(\X)^n \ar[r]^-{\Push} & \Reg^n(\X),}
\]
where the functor $\Push$ sends an $n$-tuple of regular epimorphisms with a common domain to the $n$-fold regular epimorphism induced by taking repeated pushouts. Hence the right adjoint $\Eq^n\colon {\Reg^n(\X)\to \ERel^n(\X)}$ is naturally isomorphic to the composite of right adjoints
\[
\xymatrix{\ERel^n(\X) & \ERel(\X)^n \ar[l]^-{\square} & \Reg(\X)^n \ar[l]^-{(\Eq)^n} & \Reg^n(\X), \ar[l]^-{\Forget} }
\]
so that, in particular, an $n$-fold equivalence relation in $\X$ is effective if and only if it is parallelistic. In~\cite{RVdL2} a description of the relation $(R_i)_{i\in n}$ in terms of $2^n$-matrices is given, extending the one of Subsection~\ref{Subsection Denorm}. Incidentally, this also proves that the forgetful functor~$U$ is indeed a left adjoint.

Unfortunately, when $n>2$, even in the exact Mal'tsev context, $n$-fold regular epimorphisms and $n$-cubic extensions are generally different. Parallelistic higher equivalence relations do characterise higher regular epimorphisms; we are, however, less interested in the $n$-fold regular epimorphisms themselves, since the objects occurring naturally in (co)homology are the $n$-cubic extensions. What we are thus looking for is a restriction of this equivalence to the category of $n$-cubic extensions on the right, which yields a characterisation of $n$-cubic extensions in terms of parallelistic equivalence relations satisfying an additional condition. As we shall see (Theorem~\ref{3^n iff distributive parallelistic}), the parallelistic equivalence relation $\bigboxvoid_{i\in n}R_i$ determined by the relations $(R_i)_i$ corresponds to an extension via this equivalence, precisely when a local congruence distributivity condition holds: \emph{joins of those relations distribute over meets} or, more precisely, if $J_0$, $J_1$, \dots, $J_k\subseteq n$ with $k\geq 1$ such that
$J_i \cap J_j=\emptyset$ for all $i\neq j$, then
\[
 \big(\bigmeet_{j\in J_0}R_j\big)\meet \bigjoin_{i=1}^k\big(\bigmeet_{j\in J_i}R_j\big)=
 \bigjoin_{i=1}^k\big( \bigmeet_{j\in J_0\cup J_i}R_j\big).
\]
For instance, when $n=3$, we find the equalities
\begin{align*}
	R_2\meet(R_0\join R_1)&=(R_2\meet R_0)\join (R_2\meet R_1)\\
	R_1\meet(R_0\join R_2)&=(R_1\meet R_0)\join (R_1\meet R_2)\\
	R_0\meet(R_1\join R_2)&=(R_0\meet R_1)\join (R_0\meet R_2),	
\end{align*}
while when $n=4$ we also find conditions such as 
\begin{align*}
	R_3\meet(R_0\join R_1\join R_2)&=(R_3\meet R_0)\join (R_3\meet R_1)\join (R_3\meet R_2)\\
	(R_3\meet R_2)\meet(R_0\join R_1)&=(R_3\meet R_2\meet R_0)\join (R_3\meet R_2\meet R_1)\\
	R_3\meet (R_2\join(R_0\meet R_1))&=(R_3\meet R_2)\join(R_3\meet R_0\meet R_1).
\end{align*}
In fact, these conditions are not generated by distributivity of any three equivalence relations chosen out of $R_0$, $R_1$, $R_2$, $R_3$---see Remark~\ref{Remark Not Generated by Binary} and Example~\ref{Complexes}.

\subsection{Back to the pointed case}
As in the case of $n=2$, we may return to the pointed setting of a semi-abelian category, and ask ourselves the question under which conditions a collection of normal subobjects $(K_i)_{0\leq i< n}$ on an object~$X$ induces a $3^n$-diagram, and thus an $n$-cubic extension, by first taking intersections, and then cokernels of those intersections. 

For instance, when $n=3$, the following situation can arise. Suppose that in Figure~\ref{Figure 3x3x3}, first the solid cube is constructed by pushing out the maps $f_i\colon{F_3\to F_3/K_i}$ along each other, then the rest of the diagram by taking kernels in all directions. Then there is a priori no reason why the square of regular epimorphisms in the back of the diagram is a pushout, which means that it is not clear whether or not the top horizontal sequence is exact on the right. So in general this procedure need not deliver a $3\times 3\times 3$-diagram.

We find (Theorem~\ref{Theorem 3^n} and Corollary~\ref{Corollary 3^n}) that this \emph{does} happen when 
\begin{equation}\label{Eq Distr Norm}
	 \big(\bigmeet_{j\in J_0}K_j\big)\meet \bigjoin_{i=1}^k\big(\bigmeet_{j\in J_i}K_j\big)=
 \bigjoin_{i=1}^k\big( \bigmeet_{j\in J_0\cup J_i}K_j\big)
\end{equation}
whenever $J_0$, $J_1$, \dots, $J_k\subseteq n$ with $k\geq 1$ such that
$J_i \cap J_j=\emptyset$ for all $i\neq j$.

\section{Higher extensions in exact Mal'tsev categories}\label{Section Exact Mal'tsev}
In this section we develop the non-pointed theory of $3^n$-diagrams in an exact Mal'tsev context.

\begin{definition}[Reflexive fork, fork, exact fork]
Let $\X$ be a category. A \defn{reflexive fork} or \defn{augmented reflexive graph} $F$ in $\X$ is a diagram
\begin{equation}\label{Diagram Fork}
	\xymatrix@!0@=3em{F_2 \ar@<1ex>@{->}[rr]^-{d} \ar@<-1ex>@{->}[rr]_-{c} && F_1 \ar@{->}[rr]^f \ar[ll]|-{e} && F_{0}}
\end{equation}
where $f\comp d=f\comp c$ and $d\comp e=1_{F_1}=c\comp e$. A reflexive fork in $\X$ may be seen as a functor $(\ThreeCat^+)^{\op}\to \X$, where $\ThreeCat^+$ is the $2$-truncation of the category $\Delta_+$ of finite ordinals and order-preserving maps. It has three objects $0$, $1$ and $2$ and arrows between them as in Diagram~\eqref{Diagram Fork}. We write $\Fork(\X)$ for the functor category $\Fun((\ThreeCat^+)^{\op},\X)$, the category of reflexive forks in $\X$ and natural transformations between them. Since there are no non-reflexive forks in this article, from now on we shall commonly drop the adjective ``reflexive'' and call an object of $\Fork(\X)$ a~\defn{fork}.

A fork $F$ is \defn{exact} when $(F_2,d,c,e)$ is the kernel relation of $f$ and $f$ is the coequaliser of~$(d,c)$. We write $\EFork(\X)$ for the full subcategory of $\Fork(\X)$ determined by the exact forks.
\end{definition}

\begin{definition}[Underlying arrow, underlying reflexive graph]
The \defn{underlying arrow} of a fork $F$ as in Diagram~\eqref{Diagram Fork} is the morphism $f$. Sending forks to their underlying arrows determines a functor $\arr\colon{\Fork(\X)\to \Arr(\X)}$.

The \defn{underlying reflexive graph} of $F$ is $(F_2,d,c,e)$. Sending forks to their underlying reflexive graphs determines a functor $\grph\colon{\Fork(\X)\to \RG(\X)}$.
\end{definition}

The functor $\arr$ has a right adjoint $\EqFork$ which takes an arrow $f$ and sends it to the fork~\eqref{Diagram Fork} where $(d,c)$ is the kernel pair of $f$. The functor $\grph$ has a left adjoint $\CoeqFork$ which takes a reflexive graph $(F_1,d,c,e)$ and sends it to the fork~\eqref{Diagram Fork} where~$f$ is the coequaliser of $(d,c)$. We find a commutative triangle of adjunctions as in~Figure~\ref{Figure Triangle}.
\begin{figure}
$\resizebox{.5\textwidth}{!}
{\xymatrix@!0@R=8em@C=8em{\RG(\X) \ar@<1ex>[rr]^-{\Coeq} \ar@<1ex>[rd]^-{\CoeqFork} \ar@{}[rr]|-{\bot} \ar@{}[rd]|-{\rtop} && \Arr(\X) \ar@<1ex>[ld]^-{\EqFork} \ar@<1ex>[ll]^-{\Eq}\\
& \Fork(\X) \ar@<1ex>[lu]^-{\grph} \ar@<1ex>[ru]^-{\arr} \ar@{}[ru]|-{\ltop}}}$
\caption{Adjunctions between reflexive graphs, forks and arrows.}\label{Figure Triangle}
\end{figure}
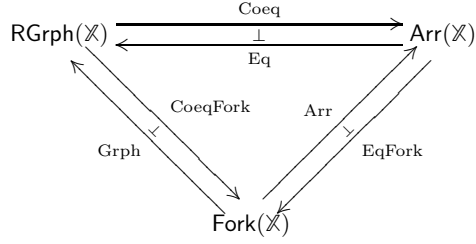

\begin{remark}
The processes that associate to a category $\X$ the category of arrows in $\X$ and the category of forks in $\X$ may be seen as functors $\Arr\colon{\mathsf{CAT}\to \mathsf{CAT}}$ and $\Fork\colon{\mathsf{CAT}\to \mathsf{CAT}}$. For instance, a functor $G\colon {\X\to \Y}$ is thus sent to the functor 
\[
\Fork(G)\colon \Fork(\X)\to \Fork (\Y)
\]
which takes a fork $F\colon{(\ThreeCat^+)^{\op}\to \X}$ in $\X$ and gives back the fork $G\comp F\colon{(\ThreeCat^+)^{\op}\to \Y}$ in $\Y$, and which acts similarly obviously on natural transformations.
\end{remark}

\begin{definition}[$n$-Fork, exact $n$-fork]
	An \defn{$1$-fork} is a fork, and when $n\geq 2$, an \defn{$n$-fork} is a fork in the category of $(n-1)$-forks. We write
\[
\Fork^n(\X)=\Fork(\Fork^{n-1}(\X))\simeq\Fun(((\ThreeCat^+)^{\op})^n,\X)
\]
for the category of $n$-forks and natural transformations between them. For $n=2$ we find the diagram in Figure~\ref{Figure n-Fork}.
\begin{figure}
$
{\xymatrix@1@!0@=2em{F_{2,2}\ar@<1ex>[rr]\ar@<-1ex>[rr]\ar@<1ex>[dd]\ar@<-1ex>[dd]_{}&&F_{1,2}\ar[ll]\ar@<1ex>[dd]^{}\ar@<-1ex>[dd]_{}\ar[rr]&&F_{0,2}\ar@<1ex>[dd]\ar@<-1ex>[dd]\\
&&&&\\
F_{2,1}\ar[uu]\ar@<1ex>[rr]\ar@<-1ex>[rr]\ar[dd]&&F_{1,1}\ar[dd]\ar[rr]\ar[uu]\ar[ll]&&F_{0,1}\ar[dd]^{} \ar[uu]\\
&&&&\\
F_{2,0}\ar@<1ex>[rr]\ar@<-1ex>[rr]_{}&&F_{1,0}\ar[ll]\ar[rr]&&F_{0,0} }}$
\caption{A $2$-fork $F$.}\label{Figure n-Fork}
\end{figure}
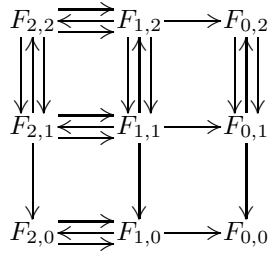

An $n$-fork $F$ is \defn{exact} when all forks in the diagram $F$ in $\X$ are exact. More precisely, this happens when for all $e\in 3^n$ and all $i\in 3$ the composite
\[
\xymatrix@1@!0@=2em{(\ThreeCat^+)^{\op}\ar@{->}[rrrrr]^-{(\alpha^+_{e,i})^{\op}}&&&&&((\ThreeCat^+)^{\op})^n\ar@{->}[rrrr]^-{F}&&&&\X
}
\]
is an exact fork, where 
\[
\alpha^+_{e,i}\colon \ThreeCat^+\to (\ThreeCat^+)^n\colon k\mapsto (e_1,\dots,e_{i-1},k,e_{i+1},\dots,e_n).
\]
An exact $n$-fork is also called a \defn{$3^n$-diagram}.
\end{definition}

\begin{definition}[Underlying $n$-fold arrow, underlying $n$-fold reflexive graph]
The \defn{underlying $n$-fold arrow} and the \defn{underlying $n$-fold reflexive graph} of an $n$-fork, as well as the respective adjoints, are defined inductively by composition of functors as follows---see also Figure~\ref{Figure n-Triangle}:
\begin{align*}
	\CoeqFork^n&=\Fork(\CoeqFork^{n-1}) \comp \CoeqFork\\
	&\qquad= \Fork^{n-1}(\CoeqFork)\comp \cdots \comp \Fork(\CoeqFork) \comp \CoeqFork,\\
	\grph^n&=\grph\comp \Fork(\grph^{n-1})=\grph\comp \Fork(\grph)\comp\cdots\comp\Fork^{n-1}(\grph),\\
	\arr^n&=\Arr(\arr^{n-1}) \comp \arr=\Arr^{n-1}(\arr)\comp \cdots \comp \Arr(\arr) \comp \arr,\\
	\EqFork^n&=\EqFork\comp \Arr(\EqFork^{n-1})\\
	&\qquad=\EqFork\comp \Arr(\EqFork)\comp\cdots\comp\Arr^{n-1}(\EqFork).
\end{align*}
\end{definition}

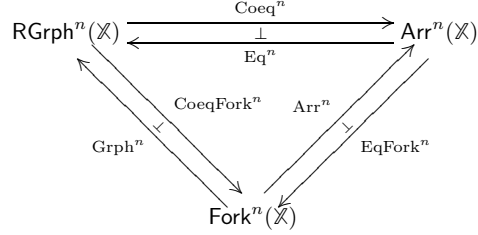
\begin{figure}
$\resizebox{.5\textwidth}{!}
{\xymatrix@!0@R=8em@C=8em{\RG^n(\X) \ar@<1ex>[rr]^-{\Coeq^n} \ar@<1ex>[rd]^-{\CoeqFork^n} \ar@{}[rr]|-{\bot} \ar@{}[rd]|-{\rtop} && \Arr^n(\X) \ar@<1ex>[ld]^-{\EqFork^n} \ar@<1ex>[ll]^-{\Eq^n}\\
& \Fork^n(\X) \ar@<1ex>[lu]^-{\grph^n} \ar@<1ex>[ru]^-{\arr^n} \ar@{}[ru]|-{\ltop}}}$
\caption{Adjunctions between $n$-fold reflexive graphs, forks and arrows.}\label{Figure n-Triangle}
\end{figure}

\begin{remark}
The triangles in Figure~\ref{Figure n-Triangle} do indeed commute: for instance,
\[
\grph^n\comp \EqFork^n=\Eq^n,
\]
because $\grph\comp \EqFork=\Eq$ by definition, and 
\begin{align*}
	\grph^n\comp \EqFork^n&=\grph\comp\Fork(\grph^{n-1})\comp \EqFork\comp\Arr(\EqFork^{n-1})\\
	&=\grph\comp \EqFork\comp\Arr(\grph^{n-1})\comp\Arr(\EqFork^{n-1})\\
	&=\Eq\comp\Arr(\Eq^{n-1})
	=\Eq^n,
\end{align*}
since in a functor category, kernel pairs are computed pointwise, so that the diagram
\[
\xymatrix@C=5em{\Arr(\Fork^{n-1}(\X)) \ar[d]_-{\EqFork} \ar[r]^-{\Arr(\grph^{n-1})} & \Arr(\RG^{n-1}(\X)) \ar[d]^-{\EqFork}\\
\Fork(\Fork^{n-1}(\X)) \ar[r]_-{\Fork(\grph^{n-1})} & \Fork(\RG^{n-1}(\X))}
\]
commutes.
\end{remark}

\begin{proposition}[\cite{EGoeVdL}]\label{Proposition exact-fork}
Suppose that $\X$ is an exact Mal'tsev category and $n\geq 0$. Let
us consider in $\Arr^n(\X)$ the exact fork
\[
\xymatrix{\Eq(f)\ar@<1.ex>@{->}[r]^-{\pi_1}\ar@<-1.ex>@{->}[r]_-{\pi_2}&B\ar[l]|-{\Delta_B}\ar@{->>}[r]^f
&A}
\]
such that $B$ is an $n$-cubic extension. Then $f$ is an $(n+1)$-cubic extension if and only if $\Eq(f)$ is an $n$-cubic extension.\noproof
\end{proposition}

\begin{remark}
	One consequence of this result is that the kernel pair of an ${(n+1)}$-cubic extension, considered as an arrow between $n$-cubic extensions in an exact Mal'tsev category $\X$, may be taken pointwise in the base category $\X$. Likewise, the (pointwise) coequaliser of a reflexive graph of $n$-cubic extensions is an $(n+1)$-cubic extension. This allows us to treat exact forks of $n$-cubic extensions pointwise as diagrams in $\X$.
\end{remark}

\begin{remark}\label{Remark Symmetry}
	For a given $n$-cube of arrows in $\X$, the extension property is \emph{symmetric}, in the sense that it is independent of the way this $n$-cube is considered as a morphism of $(n-1)$-fold arrows---see~\cite{EGoeVdL} for more details.
\end{remark}

\begin{theorem}[Denormalised $3^n$-Lemma, I]\label{3^n iff extension}
Let $\X$ be an exact Mal'tsev category. The adjunction
\[
\xymatrix{\Fork^n(\X) \ar@<1ex>[r]^-{\arr^n} \ar@{}[r]|-{\bot} & \Arr^n(\X) \ar@<1ex>[l]^-{\EqFork^n}}
\]
restricts to an adjoint equivalence ${\EFork^n(\X)\simeq \Ext^n(\X)}$. In this sense, \emph{$n$-cubic extensions are equivalent to $3^n$-diagrams}. 
\end{theorem}
\begin{proof}
We give a proof by induction. The case $n=1$ is clear. Suppose now that $\arr^{n-1}\colon{\EFork^{n-1}(\X)\to \Ext^{n-1}(\X)}$ is an equivalence. Since the square
\[
\xymatrix{\Fork(\Fork^{n-1}(\X)) \ar[d]_-{\Fork(\arr^{n-1})} \ar[r]^-{\arr} & \Arr(\Fork^{n-1}(\X)) \ar[d]^-{\Arr(\arr^{n-1})}\\
\Fork(\Arr^{n-1}(\X)) \ar[r]_-\arr & \Arr(\Arr^{n-1}(\X))}
\]
commutes and the left hand side vertical arrow restricts to an equivalence
\[
\Fork(\EFork^{n-1}(\X))\simeq \Fork(\Ext^{n-1}(\X))
\]
which is such that exact forks in $\EFork^{n-1}(\X)$ correspond to exact forks in $\Ext^{n-1}(\X)$, we only need to prove that the functor
\[
\arr\colon \Fork(\Ext^{n-1}(\X))\to \Arr(\Ext^{n-1}(\X))
\]
restricts to an equivalence $\EFork(\Ext^{n-1}(\X))\simeq\Ext^{n}(\X)$. This is an immediate consequence of Proposition~\ref{Proposition exact-fork}.
\end{proof}

Recall from Section~\ref{Section Introduction} that an $n$-fold equivalence relation is \defn{parallelistic} when it is in the replete image $\PERel^n(\X)$ of the left adjoint
\[
\xymatrix{\ERel^n(\X) \ar@<1ex>[r]^-{U} \ar@{}[r]|-{\bot} & \ERel(\X)^n, \ar@<1ex>[l]^-{\square}}
\]
where the forgetful functor $U$ takes an $n$-fold equivalence relation on $R_0$, \dots, $R_{n-1}$ and sends it to the $n$-tuple of equivalence relations $(R_0,\dots, R_{n-1})$, and the functor $\square$ takes an $n$-tuple of equivalence relations $(R_i)_{i\in n}$ and sends it to the $n$-fold equivalence relation $\bigboxvoid_{i\in n}R_i$. 

\begin{definition}[Induced $n$-fold regular epi]
	Given an $n$-tuple $(R_0,\dots, R_{n-1})$ of equivalence relations on an object $X$, we write $\Coeq_{i\in n}(R_i)$ for the \defn{induced $n$-fold regular epimorphism} $\Coeq^n(\bigboxvoid_{i\in n}R_i)$. As explained above, it may be obtained by taking successive pushouts of the coequalisers of the $R_i$.
\end{definition}

\begin{definition}[Distributivity]
A finite collection $(R_i)_{i\in n}$ of equivalence relations on an object $X$ is said to be \defn{distributive} when the following congruence distributivity condition is satisfied: if $J_0$, $J_1$, \dots, $J_k\subseteq n$ with $k\geq 1$ such that
$J_i \cap J_j=\emptyset$ for all $i\neq j$, then
\[
\big(\bigmeet_{j\in J_0}R_j\big)\meet \bigjoin_{i=1}^k\big(\bigmeet_{j\in J_i}R_j\big)=
 \bigjoin_{i=1}^k\big( \bigmeet_{j\in J_0\cup J_i}R_j\big).
\]

A parallelistic $n$-fold equivalence relation $\bigboxvoid_{i\in n}R_i$ is \defn{distributive} when the collection of equivalence relations $(R_i)_{i\in n}$ is. We write $\DPERel^n(\X)$ for the full subcategory of $\PERel^n(\X)$ determined by the distributive parallelistic $n$-fold equivalence relations.
\end{definition}

\begin{remark}
	The distributivity condition is only non-trivial when $n\geq 3$; for smaller $n$, it always holds.
\end{remark}

\begin{remark}\label{closure of Dn by finitary intersection} 
\begin{enumerate}
	\item A collection of equivalence relations $(R_i)_{i\in
n}$ is distributive if and only if for all $l\leq n$ and all $0 \leq i_0< \cdots<i_{l-1}\leq n-1$, the collection $(R_{i_j})_{j\in l}$ is.
\item If a collection of equivalence relations $(R_i)_{i\in n}$ is distributive, then for all $l\leq n$, the collection $(S_i)_{i\in l}$ where $S_i=\bigmeet_{j\in J_i}R_j$ and $ J_0$, \dots, $J_{l-1}$ are non-empty subsets of $n$ with $J_i\cap J_j=\emptyset$, $i\neq j$, is distributive as well.
\end{enumerate}
\end{remark}

Our aim is now to prove the following version of the \emph{denormalised $3^n$-Lemma}:

\begin{theorem}[Denormalised $3^n$-Lemma, II]\label{3^n iff distributive parallelistic}
Let $\X$ be an exact Mal'tsev category. The adjunction
\[
\xymatrix{\Fork^n(\X) \ar@<1ex>[r]^-{\grph^n} \ar@{}[r]|-{\bot} & \RG^n(\X) \ar@<1ex>[l]^-{\CoeqFork^n}}
\]
restricts to an adjoint equivalence ${\EFork^n(\X)\simeq \DPERel^n(\X)}$. In this sense, \emph{$3^n$-dia\-grams are equivalent to distributive parallelistic $n$-fold equivalence relations}. 
\end{theorem}

We first obtain an intermediate result (Proposition~\ref{Proposition Distr vs Ext}), which relates extensions to distributive parallelistic equivalence relations. In order to better understand why a congruence distributivity condition plays a role here, let us start with the case $n=3$ (Proposition~\ref{Proposition Distr vs Ext 3}). Recall the following result from~\cite{Carboni-Kelly-Pedicchio}:

\begin{lemma}\label{double-extension-condition}
Consider in a regular category $\X$ a commutative square
\[
\xymatrix@1@!0@=2em{A\ar@{->>}[rr]^{s}\ar@{->>}[dd]_{r}\ar@{.>>}[rrdd]^{t}&&C\ar@{->>}[dd]^{v}\\
&&\\
B\ar@{->>}[rr]_{u}&&D }
\]
of regular epimorphism where $t=u\comp r=v\comp s$ and $R=\Eq(r)$, $S=\Eq(s)$ and $T=\Eq(t)$.
Then the following conditions are equivalent:
\begin{tfae}
\item $\langle r,s\rangle\colon A\to B\times_D C$ is a regular
epimorphism;
\item $R\circ S=T=S\circ R$.
\end{tfae}
The category $\X$ is Mal'tsev category if and only if always $R\circ S=R\join S=S\circ R$. The category $\X$ is exact Mal'tsev if and only if this latter equivalence relation is always a congruence. Then the equivalent conditions {\rm (i)} and {\rm (ii)} hold if and only if the given square is a pushout.\noproof
\end{lemma}

In the same article and in~\cite{CLP}, it is explained that the lattice of equivalence relations on any object in an exact Mal'tsev category is \defn{modular}: any $R$, $S$ and $T$ on an object $X$ satisfy $(R\join S)\meet T=R\join (S\meet T)$. Equivalently, if $R\leq T$, then $(R\join S)\meet (R\join T)=R\join (S\meet T)$.

\begin{proposition}\label{Proposition Distr vs Ext 3}
	In an exact Mal'tsev category $\X$, consider an object $X$ and three equivalence relations $R_0$, $R_1$ and $R_2$ on $X$. Write $F=\Coeq_{i\in 3}(R_i)$ for the induced $3$-fold regular epimorphism. Then the following conditions are equivalent:
	\begin{tfae}
		\item $F$ is a $3$-cubic extension;
		\item $R_2\join(R_0\meet R_1)=(R_2\join R_0)\meet (R_2\join R_1)$;
		\item[(ii')] $R_1\join(R_0\meet R_2)=(R_1\join R_0)\meet (R_1\join R_2)$;
		\item[(ii'')] $R_0\join(R_1\meet R_2)=(R_0\join R_1)\meet (R_0\join R_2)$;
		\item $R_2\meet(R_0\join R_1)=(R_2\meet R_0)\join (R_2\meet R_1)$;
		\item[(iii')] $R_1\meet(R_0\join R_2)=(R_1\meet R_0)\join (R_1\meet R_2)$;
		\item[(iii'')] $R_0\meet(R_1\join R_2)=(R_0\meet R_1)\join (R_0\meet R_2)$;
		\item the collection of equivalence relations $(R_i)_{i\in 3}$ is distributive;
		\item the parallelistic $3$-fold equivalence relation $\bigboxvoid_{i\in 3}R_i$ is distributive;
		\item $R_2\square(R_0\join R_1)=(R_2\square R_0)\join (R_2\square R_1)$, as equivalence relations on $R_2$;
		\item[(vi')] $R_1\square(R_0\join R_2)=(R_1\square R_0)\join (R_1\square R_2)$, as equivalence relations on $R_1$;
		\item[(vi'')] $R_0\square(R_1\join R_2)=(R_0\square R_1)\join (R_0\square R_2)$, as equivalence relations on $R_0$.
	\end{tfae} 
\end{proposition}
\begin{proof}
We consider $F$ as the cube in Figure~\ref{Figure 3-Cube}, 
\begin{figure}
$\xymatrix@!0@=3em{&& {\cdot} \ar@{~>}[rd] \ar[rrrr] \ar@{.>}[dddd]|-{\hole} &&&& {\cdot} \ar[dddd] \\
&&& {P} \ar@{.>}[lddd] \ar@{.>}[rrru] \pullback \\
{X} \ar@{~>}[rd]|{\langle f_0,f_1\rangle} \ar[rrrr]^(.25){f_1} \ar[dddd]_{f_0} \ar@{~>}[rruu]^-{f_2} &&&& {\cdot} \ar[dddd] \ar[rruu]\\
& {\cdot} \ar@{~>}[rruu]|-{\hole} \ar@{->}[lddd] \ar@{->}[rrru] \pullback \\
&& {\cdot} \ar@{.>}[rrrr] &&&& {Q}\\\\
{\cdot} \ar[rrrr] \ar@{.>}[rruu] &&&& {\cdot} \ar[rruu]}$
\caption{The cube induced by $R_0$, $R_1$ and $R_2$.}\label{Figure 3-Cube}
\end{figure}
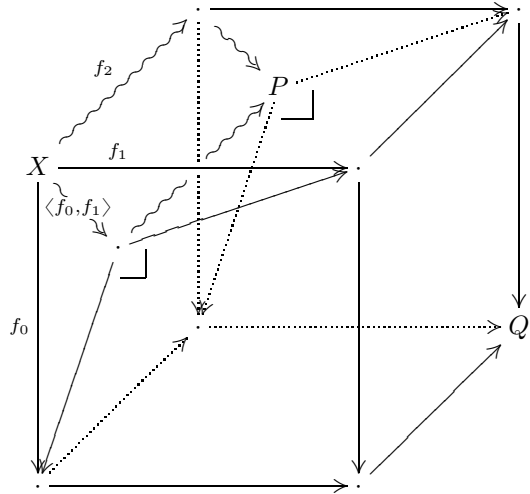
where $f_i=\Coeq(R_i)$. For this cube in $\X$ to represent an extension, we need that the induced comparison square to the pullbacks is a (regular) pushout. Note that the regular epimorphism $\langle f_0,f_1\rangle$ is the coequaliser of $R_0\meet R_1$. Now the pushout of $\Coeq(R_2\join R_0)$ and $\Coeq(R_2\join R_1)$ is $Q$, since $f_2$ is an epimorphism. Hence the comparison to the induced pullback, which is the coequaliser of $(R_2\join R_0)\meet (R_2\join R_1)$, is the composite ${X\to P}$. By Lemma~\ref{double-extension-condition} now, $F$ is a $3$-cubic extension if and only if $R_2\join(R_0\meet R_1)=(R_2\join R_0)\meet (R_2\join R_1)$, so that (i) and (ii) are equivalent. 

Since the extension property of $F$ is independent of the order of~$R_0$, $R_1$,~$R_2$, the above proof also shows that (i) is equivalent to (ii') and (i) is equivalent to (ii'').

Condition (ii') implies (iii) because
\begin{align*}
	(R_2\meet R_0)\join (R_2\meet R_1)&=((R_2\meet R_0)\join R_2)\meet ((R_2\meet R_0)\join R_1)\\
	&=R_2\meet ((R_2\join R_1)\meet (R_0\join R_1))\\
	&=R_2\meet (R_0\join R_1).
\end{align*}
Here we use modularity in the first equality, and distributivity in the second. Similarly, we find (iii') and (iii'') out of (ii) and (ii''). Now (iii), (iii') and (iii'') together are equivalent to (iv).

Condition (iii) implies (ii'') since
\begin{align*}
	(R_0\join R_1)\meet (R_0\join R_2)&=(R_0\join R_2)\meet (R_0\join R_1)\\
	&=R_0\join (R_2\meet (R_0\join R_1))\\
	&=R_0\join ((R_2\meet R_0)\join (R_2\meet R_1))\\
	&=R_0\join (R_2\meet R_1).
\end{align*} 
Likewise, also (ii) and (ii') are implied by (iv).

Conditions (iv) and (v) are equivalent by definition. For the proof that (i) and~(vi) are equivalent, take kernel pairs of the arrows in the $f_2$ direction in order to obtain the diagram in Figure~\ref{Figure Kernel Pair}. 
\begin{figure}
$
{\xymatrix@1@!0@=3em{\bigboxvoid_{i\in 3}R_i\ar@<1ex>[rr]\ar@<-1ex>[rr]\ar@<1ex>[dd]\ar@<-1ex>[dd]_{}&&R_2\square R_0\ar@<1ex>[dd]^{}\ar@<-1ex>[dd]_{}\ar[ll]\ar[rr]&&\cdot\ar@<1ex>[dd]\ar@<-1ex>[dd]\\
&&&&\\
R_2\square R_1\ar[uu]\ar@<1ex>[rr]\ar@<-1ex>[rr]\ar[dd]&&R_2\ar[uu]\ar[ll]\ar[dd]\ar[rr] \ar@{}[rrdd]|{(\ddagger)} &&\cdot\ar[dd]^{}\ar[uu]\\
&&&&\\
\cdot\ar@<1ex>[rr]\ar@<-1ex>[rr]_{}&&\cdot\ar[ll]\ar[rr]&&\cdot }}$
\caption{The kernel pair of the cube in Figure~\ref{Figure 3-Cube}.}\label{Figure Kernel Pair}
\end{figure}
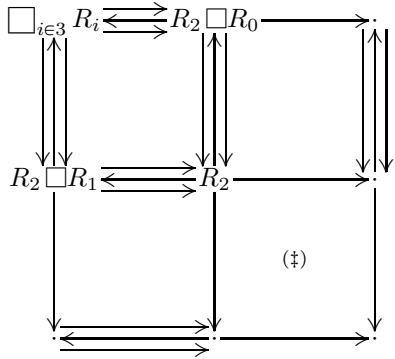
By Proposition~\ref{Proposition exact-fork}, the cube $F$ is a $3$-cubic extension if and only if the square $(\ddagger)$ is a $2$-cubic extension. By Lemma~\ref{double-extension-condition}, this happens if and only if $(R_2\square R_0)\join (R_2\square R_1)$ is the kernel pair of the diagonal in~$(\ddagger)$. However, since the front square of the cube in Figure~\ref{Figure 3-Cube} is known to be an double extension, the kernel pair in question is $R_2\square(R_0\join R_1)$. Hence (vi) holds if and only if $F$ is an extension. Similar proofs show that (vi') and (vi'') are equivalent to~(i).
\end{proof}

The case $n=3$ implies the general case (Proposition~\ref{Proposition Distr vs Ext}). To see this, we first need to prove some intermediate results.

\begin{proposition} \label{stability by intersection} 
Consider $n\geq 3$. In an exact Mal'tsev category, for any $n$-tuple $(R_i)_{i\in n}$ of equivalence relations on an object $X$, the operation $R_{n-1}\square(-)$ preserves intersections: for all $\emptyset\neq I\subseteq n-1$, we have 
\[
R_{n-1}\square (\bigmeet_{j\in I} R_j)=\bigmeet_{j\in I} (R_{n-1}\square R_j)
\]
as equivalence relations over $R_{n-1}$.
\end{proposition}
\begin{proof} 
It suffices to check this in the category $\Set$ of sets and functions.
\end{proof}
 
\begin{proposition}\label{distrib} 
Consider $n\geq 3$. Let $\X$ be an exact Mal'tsev category. Given a distributive collection of equivalence relations $(R_i)_{i\in n}$, we have
\begin{align*}
	\big(\bigjoin_{i=0}^{n-2}R_i\big) \boxvoid R_{n-1}&= \bigjoin_{i=0}^{n-2}(R_i\boxvoid R_{n-1}).
\end{align*}
Moreover, 
\begin{align*}
	\big(\bigjoin_{i=0}^k\bigmeet_{j\in J_i}R_j\big)\boxvoid R_{n-1} &=\bigjoin_{i=0}^k\bigmeet_{j\in J_i}\big(R_j\boxvoid R_{n-1}\big) 
\end{align*}
for all $J_0$, $J_1$, \dots, $J_k\subseteq n$ with $k\geq 1$ such that $J_i \cap J_j=\emptyset$ whenever $i\neq j$.
\end{proposition}
\begin{proof} 
We prove the first equation by induction. For $n=3$, the result follows from Proposition~\ref{Proposition Distr vs Ext 3}. Now for
$n> 3$, assume that
\[
\big(\bigjoin_{i=0}^{n-3}R_i\big)
\boxvoid R_{n-1}= \bigjoin_{i=0}^{n-3}(R_i\boxvoid
R_{n-1}).
\]
From the assumption that the collection $(R_i)_{i\in n}$ is distributive, we deduce
\begin{align*}
	\big((\bigjoin_{i=0}^{n-3}R_i)\join
R_{n-2}\big) \meet R_{n-1} &= \big(\bigjoin_{i=0}^{n-2}R_i\big)\meet R_{n-1}
= \bigjoin_{i=0}^{n-2}\big(R_i\meet R_{n-1}\big)\\
&= \bigjoin_{i=0}^{n-3}\big(R_i\meet R_{n-1}\big)\join \big(R_{n-2}\meet R_{n-1}\big) \\
&=\big((\bigjoin_{i=0}^{n-3}R_i)\meet R_{n-1}\big)\join
(R_{n-2}\meet R_{n-1}).
\end{align*}
Hence from Proposition~\ref{Proposition Distr vs Ext 3}, it follows that
\[
\big((\bigjoin_{i=0}^{n-3}R_i)\join
R_{n-2}\big) \boxvoid R_{n-1} =
\big((\bigjoin_{i=0}^{n-3}R_i)\boxvoid
R_{n-1}\big)\join (R_{n-2}\boxvoid R_{n-1}).
\]
Thus we see that 
\begin{align*}
 \big(\bigjoin_{i=0}^{n-2}R_i\big) \boxvoid R_{n-1} &= \big((\bigjoin_{i=0}^{n-3}R_i)\join
R_{n-2}\big) \boxvoid R_{n-1} \\
 &= \big((\bigjoin_{i=0}^{n-3}R_i)\boxvoid
R_{n-1}\big)\join ( R_{n-2}\boxvoid R_{n-1})\\
 &=\bigjoin_{i=0}^{n-3}(R_i\boxvoid
R_{n-1})\join (R_{n-2}\boxvoid R_{n-1})
 =\bigjoin_{i=0}^{n-2}(R_i\boxvoid R_{n-1}),
\end{align*}
which finishes the proof of the first claim.

Following Remark \ref{closure of Dn by finitary intersection}, the second claim is a consequence of the first: it suffices to put 
$S_i=\bigmeet_{j\in J_i}R_j$.
\end{proof}

\begin{proposition}\label{Proposition Square Preserves Distributivity}\label{closure of Dn under Dk}
For $n\geq 4$, let $(R_i)_{i\in n}$ be a distributive $n$-tuple of equivalence relations on an object $X$. Then for any $m\in n$, the $(n-1)$-tuple of equivalence relations $(R_{i}\square R_m)_{i\in n-1}$ on $R_{m}$ is still distributive.
\end{proposition}
\begin{proof}
Clearly, without any loss of generality we may assume that $m=n-1$. For any choice of subsets $J_0$, $J_1$, \dots, $J_k$ of $n-1$ with $1\leq k\leq n-2$
such that $J_i \cap J_j=\emptyset$ when $i\neq j$, we calculate
\begin{align*}
 &\big(\bigmeet_{i\in J_0}(R_i \boxvoid R_{n-1})\big)\meet \bigjoin_{i=1}^k\bigmeet_{j\in J_i}(R_j\boxvoid R_{n-1}) \\
 \overset{\text{\eqref{distrib}}}{=}&
 \big(\bigmeet_{i\in J_0}(R_i \boxvoid R_{n-1})\big)\meet \big(\big(\bigjoin_{i=1}^k\bigmeet_{j\in J_i}R_j\big)\boxvoid R_{n-1} \big)\\
 \overset{\text{\eqref{stability by intersection}}}{=}&\big(\big(\bigmeet_{i\in J_0}R_i\big) \boxvoid R_{n-1})\big)\meet \big(\big(\bigjoin_{i=1}^k\bigmeet_{j\in J_i}R_j\big)\boxvoid R_{n-1} \big)\\
 \overset{\text{\eqref{stability by intersection}}}{=}&\big( \big(\bigmeet_{i\in J_0}R_i\big)\meet\big(\bigjoin_{i=1}^k\bigmeet_{j\in J_i}R_j\big)\big)\boxvoid R_{n-1} \\
 \overset{\phantom{\text{\eqref{stability by intersection}}}}{=}& \big(\bigjoin_{i=1}^k\bigmeet_{j\in J_0\cup J_i}R_j\big)\boxvoid R_{n-1} 
 \overset{\text{\eqref{distrib}}}{=}\bigjoin_{i=1}^k\bigmeet_{j\in J_0\cup J_i}\big(R_j\boxvoid R_{n-1}\big).
\end{align*}
This completes the proof.
\end{proof}

\begin{proposition}\label{Proposition Distr vs Ext}
	In an exact Mal'tsev category $\X$, consider an object $X$ and an $n$-tuple $(R_i)_{i\in n}$ of equivalence relations on $X$, where $n\geq 3$. Write $F=\Coeq_{i\in n}(R_i)$ for the induced $n$-fold regular epimorphism. Then the following conditions are equivalent:
	\begin{tfae}
		\item $F$ is an $n$-cubic extension;
		\item the collection of equivalence relations $(R_i)_{i\in n}$ on $X$ is distributive;
		\item the parallelistic $n$-fold equivalence relation $\bigboxvoid_{i\in n}R_i$ is distributive.
	\end{tfae} 
\end{proposition}
\begin{proof}[Proof of {\rm (ii) $\Rightarrow$ (i)} and {\rm (ii) $\Leftrightarrow$ (iii)}]
Conditions (ii) and (iii) are equivalent by definition. We now use induction on $n$ to give a proof that (ii) implies (i). Recall that the case $n=3$ is covered by Proposition~\ref{Proposition Distr vs Ext 3}. So, for some $n\geq 4$, let us assume that the claim holds for $(n-1)$, and then prove it for $n$. We view the $n$-fold regular epimorphism $F=\Coeq_{i\in n}(R_i)$ as a regular epimorphism of $(n-1)$-fold regular epimorphisms $F\colon{\dom (F)\to \cod(F)}$ in the $(n-1)$-direction, so that $\dom(F)=\Coeq_{i\in n-1}(R_i)$, and $\Eq(F)=\Coeq_{i\in n-1}(R_{i}\square R_{n-1})$. 

Since the collection $(R_i)_{i\in n}$ is distributive, the induction hypothesis implies that $\dom(F)$ is an $(n-1)$-cubic extension. Hence by Proposition~\ref{Proposition exact-fork}, $F$ is an $n$-cubic extension if and only if $\Eq(F)$ is an $(n-1)$-cubic extension. By the induction hypothesis, this happens as soon as the $(n-1)$-tuple of equivalence relations $(R_{i}\square R_{n-1})_{i\in n-1}$ on $R_{n-1}$ is distributive. This follows from Proposition~\ref{Proposition Square Preserves Distributivity}, if we use that the $(n-1)$-fold equivalence relation $\Eq^{n-1}(\Eq(f))$ is sent to the $(n-1)$-tuple of equivalence relations $(R_{i}\square R_{n-1})_{i\in n-1}$ by the forgetful functor $U$. 
\end{proof}

We still need to prove that (i) implies (ii). For this we need Lemma~\ref{some obvious results about extensions}.

\begin{lemma}\label{some obvious results about extensions}
Let $\X$ be an exact Mal'tsev category and let
$(R_i)_{i\in n}$ be equivalence relations on an object $X$ of
$\X$. Suppose that $\Coeq(R_0,\dots,R_{n-1})$ is an $n$-cubic
extension. Then for every $\emptyset \subsetneq I\subseteq n$
where $|I|=k$ and $I=\{i_0,\dots,i_{k-1}\}$, the following holds:
\begin{enumerate}
 \item the $k$-cube $\Coeq(R_{i_0},\dots,R_{i_{k-1}})$ is a $k$-cubic extension.
\end{enumerate}
Furthermore, when $J\subseteq n$ such that $I\cap J=\emptyset$,
the following hold:
\begin{enumerate}
\setcounter{enumi}{1}
 \item the $k$-cube $\Coeq(R_{i_0},\dots,R_{i_{k-1}}\meet
\bigmeet_{j\in J}R_j)$ is a $k$-cubic extension;
 \item the $k$-cube $\Coeq(R_{i_0},\dots,R_{i_{k-1}}\join
\bigjoin_{j\in J}R_j)$ is a $k$-cubic extension.
\end{enumerate}
Now assume that $n\geq 4$, and let $J_0$,
$J_1$, \dots, $J_{k-1}$ be non-empty subsets of $n-1$ with $2\leq k\leq n-1$ such that $J_i\cap J_j=\emptyset$ for $i\neq j$. We have:
\begin{enumerate}
\setcounter{enumi}{3}
\item the $k$-cube $\Coeq(\bigmeet_{j\in
J_0}R_{j},\dots,\bigmeet_{j\in J_k}R_{j})$ is a
$k$-cubic extension;
\item the $k$-cube $\Coeq(\bigmeet_{j\in
J_0}R_{j},\dots,\bigmeet_{j\in I\cap J_l}R_{j},\dots,
\bigmeet_{j\in J_k}R_{j})$ is a $k$-cubic extension, for any $l\in k$ and each $I\subseteq n$ such that $I\cap J_i=\emptyset$;
\item the $k$-cube $\Coeq(\bigmeet_{j\in
J_0}R_{j},\dots,\big(\bigmeet_{j\in
J_l}R_{j}\big)\join \big(\bigmeet_{j\in I}R_{j}\big)
,\dots, \bigmeet_{j\in J_k}R_{j})$ is a $k$-cubic
extension for any $l\in k$ and each $I\subseteq n$ such that $I\cap J_i=\emptyset$.
\end{enumerate}
\end{lemma}
\begin{proof}
For the proof of (1) it suffices to show that the $(n-1)$-fold regular epimorphism $\Coeq(R_0,\dots,R_{n-2})$ is an $(n-1)$-cubic extension. The result then follows by induction and the symmetry in the concept of an extension (see Remark~\ref{Remark Symmetry}). Considering the $n$-cube $\Coeq(R_0,\dots,R_{n-1})$ as an arrow in the $n-1$ direction---so that $R_{n-1}$ occurs in its kernel pair---we find $\Coeq(R_0,\dots,R_{n-2})$ as is domain. Since the domain of an $n$-extension is an $(n-1)$-cubic extension, we obtain the result.

For the proofs of (2) and (3) we may without any loss of generality assume that $I=n-1$ and $J=\{n-1\}$; the full statements then follow from (1) and induction on the size of $J$.

So for (2) we must show that the $(n-1)$-cube $\Coeq(R_0,\dots,R_{n-2}\meet R_{n-1})$ is an extension. Notice that we may consider the $n$-cube $\Coeq(R_0,\dots,R_{n-1})$ as a square of $(n-2)$-cubes, as follows:
\[
\xymatrix@=3em{\cdot \ar[r]^-{\Coeq(R_0,\dots,R_{n-3},R_{n-2})} \ar[d]_-{\Coeq(R_0,\dots,R_{n-3},R_{n-1})} & \ar[d] \cdot\\
\cdot \ar[r] & \cdot}
\]
By definition of an $n$-extension, all the arrows in this square are $(n-1)$-cubic extensions. Furthermore, the comparison to the induced pullback is an $(n-1)$-cubic extension as well---and this is precisely the $(n-1)$-cube $\Coeq(R_0,\dots,R_{n-2}\meet R_{n-1})$.

Similarly, for (3) we have to prove that the $(n-1)$-cube
\[
\Coeq(R_0,\dots,R_{n-2}\join R_{n-1})
\]
is an $(n-1)$-cubic extension. We use the same interpretation of the $n$-cube $\Coeq(R_0,\dots,R_{n-1})$ as above. This time, it suffices to notice that the $(n-1)$-cube $\Coeq(R_0,\dots,R_{n-2}\join R_{n-1})$ is the diagonal in the square---which is an $(n-1)$-cubic extension, because extensions compose.

(4), (5) and (6) follow from (2) and (3) by induction.
\end{proof}

\begin{proof}[Proof of {\rm (i) $\Rightarrow$ (ii)} in Proposition~\ref{Proposition Distr vs Ext}]
We shall use induction on $n$. Assume
that the result is true for $k<n$ and that
$F$ is an $n$-cubic extension. Consider $J_0$,
$J_1$, \dots, $J_k\subseteq$ of $n-1$ with $1\leq k\leq n-2$ such that $J_i
\cap J_j=\emptyset$ for $i\neq j$. We need to distinguish two cases:
\begin{itemize}
	\item $ n-1\notin \bigcup_{i=0}^kJ_i$: in this
case, the result follows by the induction hypothesis;
\item $ n-1\in \bigcup_{i=0}^kJ_i$: here we have some work to do.
\end{itemize}
Let $S_i= \bigwedge_{j\in J_i}R_j$. For any $k< n$, the $(k-1)$-fold arrow
\[
\Coeq(S_0, (S_1\join S_2),S_3,\dots, (S_{k-1}
\join S_k)),
\]
is a $(k-1)$-cubic extension by Lemma \ref{some
obvious results about extensions}. Hence by
induction, the collection $(S_i)_{i\in k}$ is distributive. We see that
\begin{align*}
 &S_0\meet \big(S_1\join S_2\join S_3\join \cdots \join S_{k-1}\join
 S_k\big)\\
 =\; &(S_0\meet \big((S_1\join S_2)\join S_3\join \cdots \join (S_{k-1}\join
 S_k)\big)\\
 =\;&\big(S_0\meet (S_1\join S_2)\big)\join(S_0\meet S_3)\join \cdots \join \big(S_0\meet (S_{k-1}\join
 S_k)\big)\\
 =\; &(S_0\meet S_1)\join(S_0\meet S_2)\join(S_0\meet S_3)\join \cdots \join (S_0\meet S_{k-1})\join (S_0\meet S_k),
\end{align*}
which finishes the proof.
\end{proof}

\begin{proof}[Proof of Theorem~\ref{3^n iff distributive parallelistic}]
	This is an immediate consequence of Proposition~\ref{Proposition Distr vs Ext} combined with Theorem~\ref{3^n iff extension}.
\end{proof}

\begin{theorem}[Denormalised $3^n$-Lemma, III]\label{3^n overview}
Let $\X$ be an exact Mal'tsev category and let $F$ be an $n$-fork in $\X$. Then the following conditions are equivalent:
\begin{tfae}
	\item $F$ is a $3^n$-diagram: all forks in the diagram $F$ in $\X$ are exact;
	\item $F=\EqFork^n(\arr^n(F))$ and $\arr^n(F)$ is an $n$-cubic extension;
	\item $F=\CoeqFork^n(\grph^n(F))$ and $\grph^n(F)$ is a distributive parallelistic $n$-fold equivalence relation;
	\item $F=\EqFork^n(\arr^n(F))$ and $F=\CoeqFork^n(\grph^n(F))$.
\end{tfae}
\end{theorem}
\begin{proof}
The equivalence between (i) and (ii) is Theorem~\ref{3^n iff extension}, and (i) and (iii) are equivalent via Theorem~\ref{3^n iff distributive parallelistic}. (ii) and (iii) together imply (iv). (iv) implies (i), because for each $e\in 3^n$ and $i\in 3$, the composite
\[
\xymatrix@1@!0@=2em{(\ThreeCat^+)^{\op}\ar@{->}[rrrrr]^-{(\alpha^+_{e,i})^{\op}}&&&&&((\ThreeCat^+)^{\op})^n\ar@{->}[rrrr]^-{F}&&&&\X
}
\]
is an exact fork: its underlying reflexive graph is the kernel pair of its underlying arrow, while conversely, its underlying arrow is the coequaliser of its underlying reflexive graph.
\end{proof}

We find the following variation on Proposition~\ref{Proposition Distr vs Ext}:

\begin{corollary}
In an exact Mal'tsev category $\X$, consider an object $X$ and an $n$-tuple $(R_i)_{i\in n}$ of equivalence relations on $X$. Write $F=\EqFork^n(\Coeq_{i\in n}(R_i))$ for the $n$-fork in $\X$, induced by first taking pushouts of coequalisers, then taking kernel pairs. The $n$-fork $F$ is a $3^n$-diagram if and only if the $n$-tuple of equivalence relations $(R_i)_{i\in n}$ is distributive.\noproof
\end{corollary}

\begin{remark}\label{Remark Not Generated by Binary}
We can use this to see that the distributivity conditions are generally not generated by distributivity of binary joins over meets (that is to say, the conditions which hold when we choose $3$ out of the given $n$ equivalence relations). This would, for instance, imply that a $4$-fold regular epimorphism of which all faces are $3$-cubic extensions is always a $4$-cubic extension---which is false in general, even in abelian categories. We may for instance consider any bounded below chain complex of abelian groups~$C$
\[
\xymatrix{\cdots \ar[r] & C_3 \ar[r] & C_2 \ar[r] & C_1 \ar[r] & C_0 \ar[r] & C_{-1} \ar[r] & 0 \ar[r] & 0 \ar[r] & \cdots}
\] 
which is exact in all degrees below $C_2$ but not in $C_2$ itself. Then the augmented simplicial abelian group corresponding to it via the Dold-Kan equivalence truncates to a $4$-fold regular epimorphism which is not a $4$-cubic extension, even though all of its faces are $3$-cubic extensions. 

See~\cite{EGoeVdL} for more on the relationships between simplicial objects, their homology and the higher extension condition. See Example~\ref{Complexes} for a different argument involving a concrete example in the category of abelian groups.
\end{remark}

\section{The \texorpdfstring{$3^n$}{3n}-Lemma in semi-abelian categories}\label{Section Semiabelian}

We now give an interpretation of Theorem~\ref{3^n overview} in the context of a semi-abelian category, where exact forks are equivalent to short exact sequences, and so the $3^n$-Lemma takes a more familiar shape. 

From now on, we assume that $\X$ is semi-abelian. In particular, it is still exact Mal'tsev, so that the results of the previous section apply. But it is also pointed and protomodular, so that the ordinary definition of a short exact sequence makes sense.

\begin{definition}[$n$-Sequence, $3^n$-diagram, $n$-extension]
Let $\ThreeCat$ denote the category $ 0 \rightarrow 1 \rightarrow 2$. For $ n\geq 1$, the category $ \ThreeCat^n$ has the initial object $ i_n\DefEq (0,\dots,0)$ and the terminal object $ t_n\DefEq (2,\dots ,2)$. Moreover, it has an embedding $ \alpha_{e,i}\from \ThreeCat\to \ThreeCat^n$ parallel to the $i$-th coordinate axis, for each object $e$ whose $i$-th coordinate is $0$.
	
Now, given objects $X$ and $A$ in $\X$, an \defn{$n$-sequence under $A$ and over~$X$} in~$\X$ is a functor $ E\from (\ThreeCat^n)^{\op}\to \X$ which sends $ i_n$ to $X$, $ t_n$ to $A$. We write $\Seq^n(\X)$ for the functor category $\Fun(\ThreeCat^n)^{\op}, \X)$: the category of $n$-sequences and natural transformations between them.

An \defn{$3^n$-diagram} or \defn{$n$-extension}~\cite{PVdL1, RVdL2} under $A$ and over $X$ is an $n$-sequence~$E$ such that each composite below is a short exact sequence:
$$
\xymatrix@R=5ex@C=3em{
\ThreeCat^{\op} \ar[r]^-{(\alpha_{e,i})^{\op}} &
 (\ThreeCat^n)^{\op} \ar[r]^-{E} &
 \X.
}
$$
\end{definition}

For example, a $1$-extension under $A$ and over $X$ is just a short exact sequence $ A=E_2 \to E_1 \to E_0=X$. A $2$-extension under $A$ and over $X$ is a $3\times 3$-diagram, in which each row and column is short exact:
$$
\xymatrix@R=5ex@C=3em{
A=E_{2,2} \ar@{{ |>}->}[r] \ar@{{ |>}->}[d] &
 E_{1,2} \ar@{-{ >>}}[r] \ar@{{ |>}->}[d] &
 E_{0,2} \ar@{{ |>}->}[d] \\
E_{2,1} \ar@{{ |>}->}[r] \ar@{-{ >>}}[d] &
 E_{1,1} \ar@{-{ >>}}[r] \ar@{-{ >>}}[d] &
 E_{0,1} \ar@{-{ >>}}[d] \\
E_{2,0} \ar@{{ |>}->}[r] &
 E_{1,0} \ar@{-{ >>}}[r] &
 X=E_{0,0}.
}
$$
More generally, an $n$-extension is an $n$-fold short exact sequence: a ``normalised'' $3^n$-diagram, a pointed version of the concept of an exact $n$-fork. This is made precise with Proposition~\ref{Proposition (de)normalisation}. First we give an overview of some functors which occur naturally in this context.

\begin{definition}[Epi part, mono part]
The \defn{epi part} of a $1$-sequence $E$ is the morphism ${E_1\to E_0}$. Note that, for non-exact $1$-sequences, this morphism need not be an epimorphism. Sending $1$-sequences to their epi part determines a functor $\epi\colon{\Seq^1(\X)\to \Arr(\X)}$.

The \defn{mono part} of a $1$-sequence $E$ is the morphism ${E_2\to E_1}$. Note that, for non-exact $1$-sequences, this morphism need not be a monomorphism. Sending $1$-sequences to their mono part determines a functor $\mono\colon{\Seq^1(\X)\to \Arr(\X)}$.

These two functors naturally extend to functors 
\[
\epi^n, \mono^n\colon{\Seq^n(\X)\to \Arrn(\X)}.
\]
\end{definition}

Just like for forks (Figure~\ref{Figure Triangle} and Figure~\ref{Figure n-Triangle}) we find a commutative triangle of adjunctions (Figure~\ref{Figure n-Triangle Pointed}).
\begin{figure}
$\resizebox{.5\textwidth}{!}
{\xymatrix@!0@R=8em@C=8em{\Arr^n(\X) \ar@<1ex>[rr]^-{\Coker^n} \ar@<1ex>[rd]^-{\CokerSeq^n} \ar@{}[rr]|-{\bot} \ar@{}[rd]|-{\rtop} && \Arr^n(\X) \ar@<1ex>[ld]^-{\KerSeq^n} \ar@<1ex>[ll]^-{\K^n}\\
& \Seq^n(\X) \ar@<1ex>[lu]^-{\mono^n} \ar@<1ex>[ru]^-{\epi^n} \ar@{}[ru]|-{\ltop}}}$
\caption{Adjunctions between $n$-sequences and their epi and mono parts.}\label{Figure n-Triangle Pointed}
\end{figure}
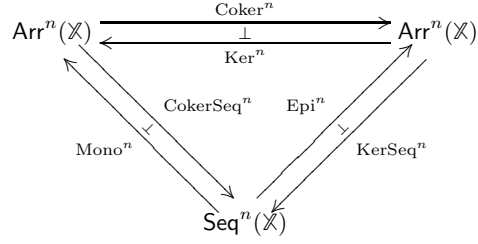
The adjunction between kernels and cokernels lifts to the level of $n$-fold arrows; completing an $n$-fold arrow to an $n$-sequence by taking cokernels defines a left adjoint to the forgetful functor $\mono^n$; and dually, the forgetful functor $\epi^n$ is left adjoint to the functor $\KerSeq^n$ which takes an $n$-cube and completes it to an $n$-sequence by taking kernels. 

\begin{definition}[$n$-Fold normal monomorphism]
	An $n$-cube in $\X$ is called an \defn{$n$-fold normal monomorphism} when all of its arrows are normal monomorphisms, and all of its squares are pullback squares. We write $\NMono^n(\X)$ for the full subcategory of $\Arrn(\X)$ determined by the $n$-fold normal monomorphisms.
\end{definition}

Any $n$-tuple of normal monomorphisms induces an $n$-fold normal monomorphism by repeated pullbacks of the given normal monomorphisms along each other. 

\begin{proposition}[The (de)normalisation process]\label{Proposition (de)normalisation}
In a semi-abelian category $\X$, the replete image $\PERel^n(\X)$ of the functor $\Eq^n$ (Figure~\ref{Figure n-Triangle}) and the replete image $\NMono^n(\X)$ of the functor $\K^n$ (Figure~\ref{Figure n-Triangle Pointed}) are equivalent categories.
\end{proposition}
\begin{proof}
When $n=1$ this is just the equivalence between (effective) equivalence relations and normal monomorphisms, since both are equivalent to the category~$\Reg(\X)$ of regular (=~normal) epimorphisms. As explained in Subsection~\ref{Subsection Higher Arrows}, the former equivalence readily extends to arbitrary degrees: $\PERel^n(\X)\simeq \Reg^n(\X)$. We prove its companion $\NMono^n(\X)\simeq \Reg^n(\X)$. 
	
We view the functor $\Coker^n\colon{\Arr^n(\X)\to \Arr^n(\X)}$ as the composite of left adjoints
\[
\xymatrix{\Arr^n(\X) \ar[r]^-{\CoForget} & \Arr(\X)^n \ar[r]^-{(\Coker)^n} & \Reg(\X)^n \ar[r]^-{\Push} & \Reg^n(\X)}
\]
and the functor $\K^n\colon {\Arr^n(\X)\to \Arr^n(\X)}$ as the composite of right adjoints
\[
\xymatrix{\NMono^n(\X) & \NMono(\X)^n \ar[l]^-{\Pull} & \Arr(\X)^n \ar[l]^-{(\K)^n} & \Arr^n(\X). \ar[l]^-{\Forget} }
\]
Here the functor
\[
\CoForget\colon{\Arr^n(\X)\to \Arr(\X)^n}
\]
send an $n$-cube to the underlying $n$-tuple of arrows with a common codomain, while its right adjoint $\Pull\colon{\Arr(\X)^n\to\Arr^n(\X)}$ sends an $n$-tuple of normal monomorphisms with a common codomain to the $n$-cube obtained by taking successive pullbacks. (Co)restricting the adjunction $\Coker^n\dashv \K^n$ to its replete images yields the needed equivalence $\NMono^n(\X)\simeq \Reg^n(\X)$.
\end{proof}

\begin{theorem}[$3^n$-Lemma]\label{Theorem 3^n}
Let $\X$ be a semi-abelian category and consider a functor $ E\from (\ThreeCat^n)^{\op}\to \X$. Then the following conditions are equivalent:
\begin{tfae}
	\item $E$ is an $3^n$-diagram: each composite $E\comp (\alpha_{e,i})^{\op}$ in the $n$-sequence $E$ is a short exact sequence in $\X$;
	\item $E=\KerSeq^n(\epi^n(E))$ and $\epi^n(E)$ is an $n$-cubic extension;
	\item $E=\CokerSeq^n(\mono^n(E))$ and $\mono^n(E)$ is induced by a distributive $n$-tuple of normal monomorphisms;
	\item $E=\KerSeq^n(\epi^n(E))$ and $E=\CokerSeq^n(\mono^n(E))$.
\end{tfae}
\end{theorem}
\begin{proof}
Via Proposition~\ref{Proposition (de)normalisation}, these equivalences correspond one by one to the equivalences in Theorem~\ref{3^n overview}.
\end{proof}

The equivalence between, (i), (ii) and (iii) may be unpacked as follows:

\begin{corollary}\label{Corollary 3^n}
A collection of $n$ normal subobjects $(K_i)_{0\leq i< n}$ on an object $X$ induces a $3^n$-diagram, and thus an $n$-cubic extension, by first taking intersections, and then cokernels of those intersections, if and only if the equality
\[
 \big(\bigmeet_{j\in J_0}K_j\big)\meet \bigjoin_{i=1}^k\big(\bigmeet_{j\in J_i}K_j\big)=
 \bigjoin_{i=1}^k\big( \bigmeet_{j\in J_0\cup J_i}K_j\big)
\]
holds whenever $J_0$, $J_1$, \dots, $J_k\subseteq n$ with $k\geq 1$ such that
$J_i \cap J_j=\emptyset$ for all $i\neq j$.\noproof	
\end{corollary}

\begin{example}\label{Arithmetical rings}
	In~\cite{Fuchs}, a ring is called \defn{arithmetical} when all of its ideals satisfy the distributivity condition. This is a well-established concept; for instance, in~\cite{Jensen} it is shown that an integral domain is arithmetical if and only if it is a so-called \emph{Pr\"ufer ring}.

Consider an $n$-sequence $E$ in the (semi-abelian) category of rings (with or without unit), such that the middle object $E_{1,\dots,1}$ in this sequence is arithmetical. Then for~$E$ to be a $3^n$-diagram, it suffices that $E=\CokerSeq^n(\mono^n(E))$ and each commutative square in the $n$-cube $\mono^n(E)$ is a pullback of normal monomorphisms. In other words, any $n$-tuple of ideals in an arithmetical ring induces a $3^n$-diagram, and this characterises the concept of an arithmetical ring.
\end{example}

\begin{example}\label{Locally cyclic groups}
A group is called \defn{locally cyclic} when all of its finitely generated subgroups are cyclic. This happens---see~\cite{Ore1, Ore2} or~\cite{Hall}---if and only if the lattice of all subgroups of the given group is distributive. So, any $n$-tuple of normal subgroups of a locally cyclic group induces a $3^n$-diagram, and this characterises the concept of a locally cyclic \emph{abelian} group.
\end{example}

\begin{example}\label{Complexes}
The abelian (additive) group of complex numbers $\C$ is well known not to be locally cyclic. 

We consider the solutions of the equation $x^3 - 1 = 0$ and call them $1$, $a$ and $a^2$. In particular, we have that $1 + a + a^2 = 0$. In the lattice of subgroups of $\C$, we single out those generated by $1$, $a$ and $a^2$. We notice that the meet of any two of them is $0$, while the join of any two is equal to $\langle1\rangle\join\langle a\rangle= \langle1\rangle\join\langle a^2\rangle=\langle a\rangle\join\langle a^2\rangle$, since each of $1$, $a$ and $a^2$ can be written as a $\Z$-linear combination of the two others. Thus it is easy to see that these three subgroups of $\C$ do not distribute. Hence they do not generate a $3\times 3 \times 3$-diagram with $\C$ in the centre.

Let us now consider the subgroups $\langle1\rangle$, $\langle2a\rangle$, $\langle3a\rangle$ and $\langle a^2\rangle$ of $\C$. It is readily checked by hand that any choice of three of those forms a distributive collection. Yet they do not form a distributive quadruple. Indeed, 
\[
\bigl((\langle2a\rangle\meet \langle3a\rangle)\join \langle1\rangle\bigr)\meet \langle a^2\rangle = (\langle6a\rangle\join \langle1\rangle)\meet \langle a^2\rangle = \langle6a^2\rangle,
\]
while 
\[
\bigl((\langle2a\rangle\meet \langle3a\rangle)\meet \langle a^2\rangle\bigr)\join (\langle1\rangle\meet \langle a^2\rangle)=(\langle6a\rangle\meet \langle a^2\rangle)\join (\langle1\rangle\meet \langle a^2\rangle) = 0\join 0 = 0.
\]
So, in the $4$-sequence generated by those subobjects of $\C$, any $3$-sequence whose ``middle object'' is $\C$ is a  $3\times 3 \times 3$-diagram. However,  the entire $4$-sequence itself fails to be a $3\times 3 \times 3\times 3$-diagram.
\end{example}

%
%
%
%
%
%

\section{Final remarks}\label{Section Final}

\subsection{Yoneda extensions}
In some sense, a $3^n$-diagram is a non-abelian replacement for the concept of a Yoneda extension~\cite{Yoneda-Exact-Sequences}. In the context of an abelian category, the two are equivalent via (a truncated version of) the Dold-Kan correspondence~\cite{Dold-Puppe,EGoeVdL}. In a semi-abelian context, $3^n$-diagrams occur in the interpretation of the derived functors of $\Hom(-,A)\colon{\X^{\op}\to \Ab}$ for any abelian object $A$ in~$\X$; see~\cite{RVdL2, PVdL1}.

\subsection{Aspherical augmented simplicial objects}
Recall from~\cite{EGoeVdL} that an augmented simplicial object $S$ in a semi-abelian category $\X$ is aspherical (i.e., all of its homology objects vanish) if and only if for every $n\geq 0$, the $(n+1)$-fold arrow induced by the $n$-truncation of $S$ is an $(n+1)$-cubic extension. From the above analysis it follows right away that if we write $k_i\colon{K_i\to X=S_n}$ for the kernel of $\del_i\colon S_n\to S_{n-1}$, then $S$ is aspherical precisely when \eqref{Eq Distr Norm} holds whenever it makes sense.

\subsection{Non-effective higher equivalence relations}
Non-effective higher equivalence relations exist, and are in fact quite common. Let us consider the case $n=2$. The normalisation of a double equivalence relation is a commutative square of normal monomorphisms. The original double equivalence relation is parallelistic if and only if its normalisation is a pullback square. So whenever in a lattice of normal subobjects we have $K\normal M\normal X$, $K\normal N\normal X$ such that $K\neq M\meet N$, we find an example of a non-effective double equivalence relation. Clearly such examples may be constructed easily, even in the abelian or in the arithmetical case.

\subsection{More general contexts}
As remarked in the introduction, the $3\times 3$-Lemma of~\cite{Bourn2001} does not only admit a non-pointed generalisation to regular Mal'tsev categories such as~\cite{Bourn2003}. This result extends at least to regular Goursat categories~\cite{Lack, ZJanelidze-Snake, Gran-Rodelo}, and the pointed and unpointed cases can be treated in a single framework~\cite{GJR}. Further extensions to ``relative'' contexts exist~\cite{Tamar-Janelidze-Thesis, Tamar_Janelidze} or~\cite{BM-3x3}. Starting from Subsection~\ref{Subsection Higher Arrows}, we restricted ourselves to exact Mal'tsev categories, essentially for the sake of simplicity. We believe that the results of this article may be similarly generalised, and hope such generalisations will be developed in the near future. 


\begin{thebibliography}{10}

\bibitem{Borceux-Bourn}
F.~Borceux and D.~Bourn, \emph{Mal'cev, protomodular, homological and
  semi-abelian categories}, Math. Appl., vol. 566, Kluwer Acad. Publ., 2004.

\bibitem{Bourn2001}
D.~Bourn, \emph{{$3\times 3$} {L}emma and protomodularity}, J.~Algebra
  \textbf{236} (2001), 778--795.

\bibitem{Bourn2003}
D.~Bourn, \emph{The denormalized {$3\times 3$} lemma}, J.~Pure Appl. Algebra
  \textbf{177} (2003), 113--129.

\bibitem{BM-3x3}
D.~Bourn and A.~Montoli, \emph{The {$3\times 3$} lemma in the
  {$\Sigma$}-{M}al'tsev and {$\Sigma$}-protomodular settings. {A}pplications to
  monoids and quandles}, preprint {\texttt{arXiv:1801.09104}}, 2018.

\bibitem{Carboni-Kelly-Pedicchio}
A.~Carboni, G.~M. Kelly, and M.~C. Pedicchio, \emph{Some remarks on {M}altsev
  and {G}oursat categories}, Appl. Categ. Structures \textbf{1} (1993),
  385--421.

\bibitem{CLP}
A.~Carboni, J.~Lambek, and M.~C. Pedicchio, \emph{Diagram chasing in {M}al'cev
  categories}, J.~Pure Appl. Algebra \textbf{69} (1991), 271--284.

\bibitem{Dold-Puppe}
A.~Dold and D.~Puppe, \emph{Homologie nicht-additiver {F}unktoren.
  {A}nwendungen}, Ann. Inst. Fourier (Grenoble) \textbf{11} (1961), 201--312.

\bibitem{EGJVdL}
T.~Everaert, J.~Goedecke, T.~Janelidze-Gray, and T.~Van~der Linden,
  \emph{Relative {Mal'tsev} categories}, Theory Appl.~Categ. \textbf{28}
  (2013), no.~29, 1002--1021.

\bibitem{EGoeVdL}
T.~Everaert, J.~Goedecke, and T.~Van~der Linden, \emph{Resolutions, higher
  extensions and the relative {M}al'tsev axiom}, J.~Algebra \textbf{371}
  (2012), 132--155.

\bibitem{EGVdL}
T.~Everaert, M.~Gran, and T.~Van~der Linden, \emph{Higher {H}opf formulae for
  homology via {G}alois {T}heory}, Adv.~Math. \textbf{217} (2008), no.~5,
  2231--2267.

\bibitem{Fuchs}
L.~Fuchs, \emph{{\"U}ber die {I}deale arithmetischer {R}inge}, Comment. Math.
  Helv. \textbf{23} (1949), no.~1, 334--341.

\bibitem{GJR}
M.~Gran, Z.~Janelidze, and D.~Rodelo, \emph{{$3\times 3$} lemma for star-exact
  sequences}, Homology, Homotopy Appl. \textbf{14} (2012), no.~2, 1--22.

\bibitem{Gran-Rodelo}
M.~Gran and D.~Rodelo, \emph{A new characterisation of {G}oursat categories},
  Appl. Categ. Structures \textbf{20} (2012), no.~3, 229--238.

\bibitem{Hall}
M.~Hall, \emph{The theory of groups}, Macmillan, 1959.

\bibitem{Janelidze-Marki-Tholen}
G.~Janelidze, L.~M{\'a}rki, and W.~Tholen, \emph{Semi-abelian categories},
  J.~Pure Appl. Algebra \textbf{168} (2002), no.~2--3, 367--386.

\bibitem{Tamar_Janelidze}
T.~Janelidze, \emph{Relative homological categories}, J.\ Homotopy Relat.\
  Struct. \textbf{1} (2006), no.~1, 185--194.

\bibitem{Tamar-Janelidze-Thesis}
T.~Janelidze, \emph{Foundation of relative non-abelian homological algebra}, Ph.D.
  thesis, University of Cape Town, 2009.

\bibitem{ZJanelidze-Snake}
Z.~Janelidze, \emph{The pointed subobject functor, {$3\times 3$} lemmas, and
  subtractivity of spans}, Theory Appl. Categ. \textbf{23} (2010), no.~11,
  221--242.

\bibitem{Jensen}
Ch.~U. Jensen, \emph{On characterizations of {Pr\"ufer} rings}, Math. Scand.
  \textbf{13} (1963), 90--98.

\bibitem{Lack}
S.~Lack, \emph{The 3-by-3 lemma for regular {G}oursat categories}, Homology,
  Homotopy Appl. \textbf{6} (2004), no.~1, 1--3.

\bibitem{Ore1}
{\O}.~Ore, \emph{Structures and group theory. {I}}, Duke Math. J. \textbf{3}
  (1937), no.~2, 149--174.

\bibitem{Ore2}
{\O}.~Ore, \emph{Structures and group theory. {II}}, Duke Math. J. \textbf{4}
  (1938), no.~2, 247--269.

\bibitem{Pedicchio2}
M.~C. Pedicchio, \emph{Arithmetical categories and commutator theory}, Appl.
  Categ. Structures \textbf{4} (1996), no.~2--3, 297--305.

\bibitem{PVdL1}
G.~Peschke and T.~Van~der Linden, \emph{The {Y}oneda isomorphism commutes with
  homology}, J.~Pure Appl.\ Algebra \textbf{220} (2016), no.~2, 495--517.

\bibitem{RVdL2}
D.~Rodelo and T.~Van~der Linden, \emph{Higher central extensions and
  cohomology}, Adv. Math. \textbf{287} (2016), 31--108.

\bibitem{Yoneda-Exact-Sequences}
N.~Yoneda, \emph{On {E}xt and exact sequences}, J. Fac. Sci. Univ. Tokyo
  \textbf{1} (1960), no.~8, 507--576.

\end{thebibliography}

\end{document}